\documentclass[11pt,a4paper,draft]{amsart} 
\usepackage{amssymb,amscd,amsmath, young,mathrsfs}
\usepackage{mathrsfs, mathdots} 
\usepackage[mathcal]{eucal} 
\usepackage{float}
\usepackage[usenames]{color}
\usepackage{soul}
\usepackage{longtable}

\theoremstyle{plain}
\newtheorem{thm}{Theorem}[section]
\newtheorem{lem}[thm]{Lemma}
\newtheorem{pro}[thm]{Proposition}
\newtheorem{cor}[thm]{Corollary}
\newtheorem*{claim*}{Claim}

\newtheorem{con}[thm]{Conjecture}
\newtheorem{lemma}[thm]{Lemma}

\theoremstyle{remark}
\newtheorem{rem}[thm]{Remark}

\newtheorem{exm}[thm]{Example}

\newtheorem{dfn}[thm]{Definition}
\newtheorem*{acknowledgements}{Acknowledgements}

\numberwithin{equation}{section}
\numberwithin{table}{section}

\newcommand{\N}{\mathbb{N}}
\newcommand{\Z}{\mathbb{Z}}
\newcommand{\Q}{\mathbb{Q}}
\newcommand{\F}{\mathbb{F}}

\newcommand{\tensor}{\otimes}

\newcommand{\ol}{\overline}

\newcommand{\Gri}{\ensuremath{\mathcal{O}}}

\renewcommand{\epsilon}{\varepsilon}
\newcommand{\vphi}{\varphi_{\mathcal{A}}}
\renewcommand{\phi}{\varphi}
\renewcommand{\theta}{\vartheta}
\newcommand{\mcO}{\mathcal{O}}
\newcommand{\mcI}{\mathcal{I}}

\newcommand{\rarr}{\rightarrow}
\newcommand{\epsi}{\varepsilon}
\newcommand{\val}{\mathrm{val}}

\newcommand{\zideal}{\zeta^{\triangleleft}}

\DeclareMathOperator{\supp}{supp}
\DeclareMathOperator{\len}{len}
\DeclareMathOperator{\WO}{WO}
\DeclareMathOperator{\Adm}{Adm}
\DeclareMathOperator{\Des}{Des}

\DeclareMathOperator{\GL}{GL}

\DeclareMathOperator{\Mat}{Mat}

\def \wo {h}

\def \Iwo {\ensuremath{I^{\textup{wo}}}}

\def \bfz {{\bf 0}}
\def \bfo {{\bf 1}}

\def \bfe {{\bf e}}
\def \bff {{\bf f}}

\def \bfr {{\bf r}}
\def \bfs {{\bf s}}

\def \bfX {{\bf X}}

\def \mcD {\ensuremath{\mathcal{D}}}
\def \mcJ {\ensuremath{\mathcal{J}}}

\def \mcS {\ensuremath{\mathcal{S}}}

\def \mcI {\ensuremath{\mathcal{I}}}

\def \p {\ensuremath{\mathfrak{p}}}

\def \Zp  {\mathbb{Z}_p}

\newcommand{\gp}[1]{\mathrm{gp}(#1)}
\newcommand{\gpzero}[1]{\mathrm{gp_0}(#1)}

\author{Michael M.~Schein} \address{Department of Mathematics, Bar-Ilan University, Ramat Gan 52900, Israel}\email{mschein@math.biu.ac.il}

\author{Christopher Voll} \address{Fakult\"at f\"ur Mathematik,
  Universit\"at Bielefeld, D-33501 Bielefeld, Germany}
\email{C.Voll.98@cantab.net}



\thanks{Schein was supported by grant 2264/2010 from the
  Germany-Israel Foundation for Scientific Research and Development
  and a grant from the Pollack Family Foundation.}

\begin{document}
 \title[Normal zeta functions of Heisenberg groups over number rings
   I]{Normal zeta functions of the Heisenberg groups over number rings
   I -- the unramified case} \date{\today}

 \begin{abstract} 
 Let $K$ be a number field with ring of integers $\mathcal{O}_K$.  We
 compute the local factors of the normal zeta functions of the
 Heisenberg groups $H(\mcO_K)$ at rational primes which are unramified
 in~$K$.  These factors are expressed as sums, indexed by Dyck words,
 of functions defined in terms of combinatorial objects such as weak
 orderings.  We show that these local zeta functions satisfy
 functional equations upon the inversion of the prime.
\end{abstract}
\maketitle


\thispagestyle{empty}

\section{Introduction}
\subsection{Normal zeta functions of groups}\label{subsec:1.1}
If $G$ is a finitely generated group, then the numbers
$a_n^\vartriangleleft (G)$ of normal subgroups of $G$ of index $n$ in
$G$ are finite for all~$n \in\mathbb{N}$.  In their seminal paper
\cite{GSS/88}, Grunewald, Segal, and Smith defined the \emph{normal
  zeta function} of $G$ to be the Dirichlet generating function
\begin{equation*}
\zeta_G^{\vartriangleleft} (s) = \sum_{n = 1}^\infty a_n^\vartriangleleft (G) n^{-s}.
\end{equation*}
Here $s$ is a complex variable. If $G$ is a finitely generated
nilpotent group, then its normal zeta function converges absolutely on
some complex half-plane. In this case the Euler product decomposition
$$\zeta_{G}^{\vartriangleleft} (s) = \prod_{p \textup{ prime}}
\zeta_{G,p}^{\vartriangleleft} (s)$$ holds, where the product runs
over all rational primes, and for each prime $p$, $$
\zeta_{G,p}^{\vartriangleleft} (s) = \sum_{k = 0}^\infty
a_{p^k}^\vartriangleleft (G) p^{-ks}$$ counts normal subgroups of $G$
of $p$-power index in~$G$; cf.\ \cite[Proposition 4]{GSS/88}. The
Euler factors $\zeta^\vartriangleleft_{G,p}(s)$ are all rational
functions in $p^{-s}$; cf.\ \cite[Theorem~1]{GSS/88}.

For any ring $R$ the Heisenberg group
over $R$ is defined as
\begin{equation} \label{equ:heisenberg.abc}
H(R) = \left\{ \left( \begin{array}{ccc} 1 & a & c \\ 0 & 1 & b \\ 0 &
  0 & 1 \end{array} \right) \mid a, b, c \in R \right\}.
\end{equation} 
In this paper, we study the normal zeta functions of the Heisenberg
groups $H(\mathcal{O}_K)$, where $\mathcal{O}_K$ is the ring of
integers of a number field~$K$. The groups $H(\mcO_K)$ are finitely
generated, nilpotent of class~$2$, and torsion-free. 

Let $n = [K:\Q]$ and $g\in\N$. Given $g$-tuples $\mathbf{e} = (e_1,
\dots, e_g)\in\N^g$ and $\mathbf{f} = (f_1, \dots, f_g)\in\N^g$
satisfying $\sum_{i = 1}^g e_i f_i = n$, we say that a (rational)
prime $p$ is of \emph{decomposition type} $(\bfe,\bff)$ \emph{in $K$}
if 
\begin{equation*}
p \mathcal{O}_K = \p_1^{e_1} \cdots \p_g^{e_g},
\end{equation*}
where the $\p_i$ are distinct prime ideals in $\mcO_K$ with
ramification indices $e_i$ and inertia degrees $f_i = [\mathcal{O}_K /
  \p_i : \F_p]$ for~$i=1,\dots,g$.  Note that this notion of
decomposition type features some redundancy reflecting the absence of a
natural ordering of the prime ideals of $\mathcal{O}_K$ lying above
$p$.  One of the main results of \cite{GSS/88} asserts that the Euler
factors $\zideal_{H(\mcO_K),p}(s)$ are rational in the two parameters
$p^{-s}$ and $p$ on sets of primes of fixed decomposition
type in~$K$:

\begin{thm}\cite[Theorem~3]{GSS/88}
Given $(\bfe,\bff)\in\N^g \times \N^g$ with $\sum_{i=1}^g e_if_i=n$,
there exists a rational function $W^\vartriangleleft_{\mathbf{e},
  \mathbf{f}}(X,Y) \in \Q(X,Y)$ such that, for all number fields $K$
of degree $[K:\Q]=n$ and for all primes $p$ of decomposition type
$(\bfe,\bff)$ in $K$, the following identity holds:
$$\zeta_{H(\mathcal{O}_K),p}^{\vartriangleleft} (s) =
W^\vartriangleleft_{\mathbf{e}, \mathbf{f}}(p,p^{-s}).$$
\end{thm}

We write $\bfo$ for the vector $(1,\dots,1)\in\N^g$, all of whose
components are ones. We remark (see~\eqref{equ:linear}) that if
$H(\Z)^g$ denotes the direct product of $g$ copies of $H(\Z)$, then
for all primes $p$ we have
$$W_{\bfo,\bfo}^\vartriangleleft(p,p^{-s}) = \zideal_{H(\Z)^g,p}(s).$$

\subsection{Main results}
In Theorem~\ref{thm:main.thm.unram} we explicitly compute the
functions $W^\vartriangleleft_{\bfo, \mathbf{f}}(X,Y)$, thereby
finding the Euler factors
$\zeta^\vartriangleleft_{H(\mathcal{O}_K),p}$ at all rational primes
$p$ that are unramified in~$K$, i.e.~those for which $\bfe = 1$.  The
functions $W^\vartriangleleft_{\bfo, \mathbf{f}}(X,Y)$ are expressed
as sums, indexed by Dyck words, where each summand is a product of
functions that can be interpreted combinatorially.  We use the
explicit formulae to prove the following functional equations:

\begin{thm} \label{thm:main.intro}
Let $\bff\in\N^g$ with $\sum_{i=1}^gf_i=n$. Then
\begin{equation} \label{equ:functeq.generic}
 W^\vartriangleleft_{\bfo, \mathbf{f}}(X^{-1},Y^{-1}) = (-1)^{3n}
 X^{\binom{3n}{2}} Y^{5n} W^\vartriangleleft_{\bfo, \mathbf{f}}(X,Y).
 \end{equation}
\end{thm}

By \cite[Theorem C]{Voll/10}, the Euler factors
$\zideal_{H(\mcO_K),p}$ satisfy a functional equation upon inversion
of the parameter $p$ for all but finitely many~$p$.  However, the
methods of that paper do not determine the finite set of exceptional
primes. In general it is not known whether any functional equation
obtains at the exceptional primes.  For the Heisenberg groups, we
establish such functional equations for non-split primes in the
forthcoming paper~\cite{SV2/14}:

\begin{thm}\cite[Theorem~1.1]{SV2/14} \label{thm:main.nonsplit.intro} 
Let $e,f\in\N$ with
$ef=n$. Then
$$ W^\vartriangleleft_{(e),(f)}(X^{-1},Y^{-1}) = (-1)^{3n}
    X^{\binom{3n}{2}} Y^{5n + 2(e-1)f}
    W^\vartriangleleft_{(e),(f)}(X,Y).$$
\end{thm}
Based on Theorems~\ref{thm:main.intro} and
\ref{thm:main.nonsplit.intro} and computations of Euler factors that
we have performed in other cases for $n=4$, we conjecture the
existence of a functional equation at {\emph{all}} primes for
Heisenberg groups over number rings.

\begin{con} \label{conj:functeq}
Let $(\mathbf{e}, \mathbf{f})\in\N^g\times\N^g$ with $\sum_{i=1}^g
e_if_i=n$.  Then
\begin{equation*}
 W^\vartriangleleft_{\mathbf{e}, \mathbf{f}}(X^{-1},Y^{-1}) = (-1)^{3n}
 X^{\binom{3n}{2}} Y^{5n + \sum_{i=1}^g 2(e_i - 1)f_i}
 W^\vartriangleleft_{\mathbf{e}, \mathbf{f}}(X,Y).
\end{equation*}
\end{con}

In particular we conjecture that, for the groups $H(\mcO_K)$, the
finite set of rational primes excluded in \cite[Theorem~C]{Voll/10}
consists precisely of the primes that ramify in~$K$.  The conjectured
existence of a functional equation at all primes is remarkable, since
in general this does not hold even for groups where a functional
equation is satisfied at all but finitely many primes by
\cite[Theorem~C]{Voll/10}.

Our methods in fact allow the rational functions $W^\vartriangleleft_{\bfe,
  \bff} (X,Y)$ to be determined explicitly for {\emph{any}}
decomposition type $(\bfe, \bff)$.  However, if $g>1$ and
$\bfe\neq\bfo$, then we do not in general know how to interpret these
explicit formulae in terms of functions that are known to satisfy a
functional equation. Conjecture \ref{conj:functeq} has been verified
for all cases occurring for~$n \le 4$.

Prior to this work, the functions $\zideal_{H(\mcO_K),p}$ had been
known only in a very limited number of cases; see \cite[Section
  2]{duSWoodward/08} for a summary of the previously available
results.  In \cite[Section~8]{GSS/88} the local functions were
computed for all primes when $n = 2$ and for the inert and totally
ramified primes when $n = 3$.  The remaining cases for $n = 3$ were
computed in Taylor's thesis \cite{Taylor/01}, using computer-assisted
calculations of cone integrals; see \cite{duSG/00}.  Finally, Woodward
determined $W^\vartriangleleft_{\bfo,\bfo}(X,Y)$ for~$n = 4$.  The
numerator of this rational function is the first polynomial in
\cite[Appendix~A]{duSWoodward/08}, where it takes up nearly a full
page.  Example \ref{exm:luke} below exhibits how our method produces
this function as a sum of fourteen well-understood summands.

\subsection{Related work and open problems}
In the recent past, zeta functions associated to Heisenberg groups and
their various generalizations have often served as a test case for an
ensuing general theory. For instance, the seminal paper~\cite{GSS/88}
contains special cases of the computations done in the present paper
as examples. Similarly, Ezzat \cite{Ezzat/14} computed the
representation zeta functions of the groups $H(\Gri_K)$ for quadratic
number rings $\mathcal{O}_K$, enumerating irreducible
finite-dimensional complex representations of such groups up to twists
by one-dimensional representations.  The paper \cite{StasinskiVoll/14}
develops a general framework for the study of representation zeta
functions of finitely generated nilpotent groups. Moreover, it
generalized Ezzat's explicit formulae to arbitrary number rings and
more general group schemes.

The current paper leaves open a number of challenges. One of them is
the computation of the rational functions
$W^{\triangleleft}_{\bfe,\bff}$ for general $\bfe\in\N^g$; in the
special case $g=1$, this has been achieved in~\cite{SV2/14}. Another
one is the computation of the local factors of the \emph{subgroup zeta
  function} $\zeta_{H(\Gri_K)}(s)$ enumerating
\emph{all} subgroups of finite index in $H(\Gri_K)$. This has not even
been fully achieved for quadratic number rings~$\Gri_K$.

More generally, it is of interest to compute the (normal) subgroup
zeta functions of other finitely generated nilpotent groups, and their
behavior under base extension. We refer the reader
to~\cite{duSWoodward/08} for a comprehensive list of examples. In his
MSc~thesis \cite{Bauer/13}, Bauer has generalized many of our
results to the normal zeta functions of the higher Heisenberg groups
$H_m(\mathcal{O}_K)$ for all $m \in \N$, where $H_m$ is a centrally
amalgamated product of $m$ Heisenberg groups.  In other words, if $R$
is a ring and we view elements of $R^m$ as row vectors, and if $I_m$
denotes the $m \times m$ identity matrix, then
\begin{equation*}
H_m(R) = \left\{ \left( \begin{array}{ccc} 1 & \mathbf{a} & c \\ 0 &
  I_m & \mathbf{b}^T \\ 0 & 0 & 1 \end{array} \right) \mid \mathbf{a},
\mathbf{b} \in R^m, \, c \in R \right\}.
\end{equation*} 
The paper~\cite{KlopschVoll/09} arose from the (uncompleted) project
to compute the subgroup zeta functions $\zeta_{H_m(\Z)}(s)$.

\subsection{Structure of the proofs of the main results}
The problem of counting normal subgroups in a finitely generated
torsion-free nilpotent group of nilpotency class $2$ is known to be
equivalent to that of counting ideals in a suitable Lie ring;
cf.~\cite[Section~4]{GSS/88}. Specifically, let $Z$ be the center of
$H(\mathcal{O}_K)$; it is easy to see that this is the subgroup of
matrices satisfying $a = b = 0$ in the notation
of~\eqref{equ:heisenberg.abc}, and that it coincides with the derived subgroup
of~$H(\mcO_K)$.  Define the Lie ring
$${L} = Z \oplus\left(H(\mathcal{O}_K) / Z\right),$$ with Lie bracket
induced by commutators in the group~$H(\mathcal{O}_K)$. It is easy
to verify that $L\cong L(\mcO_K)$ where, more generally and in analogy
with~\eqref{equ:heisenberg.abc}, the Heisenberg Lie ring
over an arbitrary ring $R$ is defined as
\begin{equation*} 
L(R) = \left\{ \left( \begin{array}{ccc} 0 & a & c \\ 0 & 0 & b \\ 0 &
  0 & 0 \end{array} \right) \mid a, b, c \in R \right\},
\end{equation*} 
with Lie bracket induced from $\mathfrak{gl}_3(R)$. The \emph{ideal
  zeta function} of $L(\mcO_K)$ is the Dirichlet generating function
$$\zideal_{L(\mcO_K)}(s) = \sum_{n=1}^\infty
a_n^\vartriangleleft(L(\mcO_K))n^{-s},$$ where
$a_n^\vartriangleleft(L(\mcO_K))$ denotes the number of ideals of
$L(\mcO_K)$ of index $n$ in $L(\mcO_K)$. This zeta function, too,
satisfies an Euler product decomposition, of the form
$$\zideal_{L(\mcO_K)}(s) = \prod_{p \mbox{ } \mathrm{prime}}\zideal_{L(\mcO_K),p}(s)= \prod_{p \mbox{ } \mathrm{prime}}\sum_{k=0}^\infty
a_{p^k}^\vartriangleleft(L(\mcO_K))p^{-ks}.$$
By the remark following \cite[Lemma~4.9]{GSS/88} we have, for all
primes $p$, that
$$\zeta^\vartriangleleft_{H(\mathcal{O}_K),p} =
\zideal_{L(\mcO_K),p}.$$

Now set $R_p = \mathcal{O}_K \otimes_\Z\Zp$
and $L_p=L(R_p)$ for every prime $p$. We write $L^\prime_p = [L_p, L_p]$ for the derived
subring and center of $L_p$, and denote by $\ol{L}_p$ the
abelianization $L_p / [L_p, L_p]$.  The $\Z_p$-modules underlying
$L^\prime_p$ and $\ol{L}_p$ have ranks $n$ and $2n$,
respectively. Then
$$L_p = L(R_p) \cong L_p' \oplus \ol{L}_p.$$ The Euler factor
$\zideal_{L(\mcO_K),p}$ may be identified with the ideal zeta function
$\zideal_{L_p}$ of the $\Zp$-Lie lattice $L_p$, enumerating
$\Zp$-ideals of $L_p$ of finite additive index in~$L_p$. To summarize,
the following equalities hold for all primes $p$:
\begin{equation}\label{equ:linear}
 \zideal_{H(\mcO_K),p} = \zideal_{L(\mcO_K),p} = \zideal_{L(R_p)} =
 \zideal_{L_p}.
\end{equation}
Essentially by \cite[Lemma 6.1]{GSS/88}, we have that
\begin{equation} \label{equ:intro.basic.eq} 
 \zideal_{L_p}(s) = \sum_{\ol{\Lambda}\leq_f \ol{L}_p}
 |\ol{L}_p:\ol{\Lambda}|^{-s} \sum_{[\ol{\Lambda},L_p] \leq M \leq
   L'_p} |L_p':M|^{2n-s}.
\end{equation}
Here the outer sum runs over all $\Z_p$-sublattices
$\overline{\Lambda} \leq \ol{L}_p$ of finite additive index.  We
briefly summarize our strategy for computing the right-hand side
of~\eqref{equ:intro.basic.eq}.  Let $p$ be a prime of decomposition
type $(\bfe,\bff)$ in $K$. In Lemma~\ref{lem:gentype}, we determine the
isomorphism type of the finite $p$-group
$L_p^\prime/[\overline{\Lambda}, L_p]$ for every finite-index
sublattice $\overline{\Lambda} \leq_f \ol{L}_p$. More precisely, we
associate to $\ol{\Lambda}$ an $n$-tuple $\ell =\ell(\ol{\Lambda}) =
(\ell_1, \dots, \ell_n) \in \N_0^n$ such that
$$ L_p^\prime / [\overline{\Lambda}, L_p] \simeq \Z / p^{\ell_1} \Z
\times \cdots \times \Z / p^{\ell_n} \Z.$$
Noting that the inner sum of~\eqref{equ:intro.basic.eq} depends only
on $\ell$ and not on $\overline{\Lambda}$, we proceed to evaluate the
outer sum in terms of the parameters $\ell$; cf.~Lemma~\ref{lem:l}.
By this point, we are able to transform~\eqref{equ:intro.basic.eq}
into the equation
\begin{equation*} 
  \zideal_{L_p}(s) = \left( \prod_{i=1}^g (1 - p^{-2 f_i s}) \right)
  \,\zeta^\vartriangleleft_{\Z_p^{2n}}(s) \,D^{\bfe, \bff} (p,p^{-s}),
\end{equation*}
where
\begin{equation} \label{equ:intro.dpt} D^{\bfe, \bff} (p,p^{-s}) =
  \sum_{\ell \in \mathrm{Adm}_{\bfe, \bff}} p^{-2s\sum_{i=1}^n\ell_i}
  \sum_{\mu \leq \lambda(\ell)} \alpha(\lambda(\ell), \mu;p) \mbox{
  }p^{({2n-s}){\sum_{i=1}^n\mu_i}};
\end{equation}
cf.\ Lemma~\ref{lem:rewrite}. The zeta function
$\zeta^\vartriangleleft_{\Zp^{2n}}(s)$ is well known;
cf.~\eqref{equ:ab}. We now explain the meanings of the terms
in~\eqref{equ:intro.dpt}.

The set
$\mathrm{Adm}_{\bfe, \bff} \subseteq \N_0^n$ of \emph{admissible}
$n$-tuples only depends on the decomposition type $(\bfe, \bff)$ of
$p$ in~$K$; cf.\ Definition~\ref{def:adm}.  For an $n$-tuple $\ell \in
\N_0^n$, we define $\lambda(\ell)$ to be the partition $\lambda_1 \geq
\cdots \geq \lambda_n$ obtained by arranging the components of $\ell$
in non-ascending order.  As $\ell$ runs over $\mathrm{Adm}_{\bfe,
  \bff}$, the partitions $\lambda(\ell)$ run over all the possible
elementary divisor types of commutator lattices~$[\overline{\Lambda},
  L_p] \leq L_p^\prime$.  The inner sum on the right-hand side
of~\eqref{equ:intro.dpt} runs over all partitions $\mu$ which are
dominated by $\lambda(\ell)$.  Finally, $\alpha(\lambda(\ell), \mu;
p)$ denotes the number of abelian $p$-groups of type $\mu$ contained
in a fixed abelian $p$-group of type $\lambda(\ell)$.  A classical formula
of Birkhoff expresses this number in terms of the dual partitions of
$\lambda(\ell)$ and $\mu$; see Proposition~\ref{pro:birkhoff}.

So far, everything we have said holds for all decomposition types
$(\bfe, \bff)$.  The difficulty in evaluating \eqref{equ:intro.dpt}
comes from the strong dependence of $\alpha(\lambda(\ell), \mu; p)$ on
the relative sizes of the parts of the partitions $\lambda(\ell)$
and~$\mu$. For unramified primes, we overcome this difficulty by
splitting $D^{\bfo, \bff}$ into a finite sum of more tractable
functions.  Indeed, the different ways in which the partition
$\lambda(\ell)$ can ``overlap'' the partition $\mu$ are parameterized
by Dyck words of length~$2n$; see Subsection~\ref{sec:dyckwords} for
details.  Given such a Dyck word $w$, we define a sub-sum $D_w^{\bfo,
  \bff}$ of $D^{\bfo,\bff}$ running over pairs of partitions
$(\lambda(\ell), \mu)$ whose overlap is captured by $w$, so that
$$D^{\bfo,\bff} = \sum_{w\in \mcD_{2n}} D_w^{\bfo,\bff},$$ where
$\mcD_{2n}$ is the set of Dyck words of length $2n$; see
Section~\ref{subsec:rewriting}.  The cardinality of $\mcD_{2n}$ is the
$n$-th Catalan number $\mathrm{Cat}_n = \frac{1}{n+1} \binom{2n}{n}$.

Each $D_w^{\bfo,\bff}$ can be expressed in terms of the Igusa
functions introduced in \cite{Voll/05} and their partial
generalizations defined in Subsection~\ref{subsec:weak.order}. The
latter may be interpreted as fine Hilbert series of Stanley-Reisner
rings of barycentric subdivisions of simplices.  Stanley proved that
these rational functions satisfy a functional equation upon inversion
of their variables.  We deduce that the functions $D^{\mathbf{1},
  \bff}_w$ all satisfy a functional equation whose symmetry factor is
independent of the Dyck word~$w$. This allows us to prove
Theorem~\ref{thm:main.intro}.

\begin{acknowledgements}
We are grateful to Mark Berman for bringing us together to work on
this project, to Kai-Uwe Bux for conversations about face complexes,
and to the referee for helpful comments.  We acknowledge support by
the DFG Sonderforschungsbereich 701 ``Spectral Structures and
Topological Methods in Mathematics'' at Bielefeld University.
\end{acknowledgements}

\section{Preliminaries} \label{sec:prelim}
Throughout this paper, $K$ is a number field of degree $n = [K : \Q]$
with ring of integers~$\mcO_K$. By $p$ we denote a rational prime, and
we fix the abbreviation $t = p^{-s}$.  For an integer $m \geq 1$, we
write $[m]$ for $\{ 1, 2, \dots, m \}$ and $[m]_0$ for $\{ 0, 1,
\dots, m \}$. Given integers $a,b$ with $a\leq b$, we write $[a,b]$
for $\{a,a+1,\dots,b\}$, and $]a,b]$ for $\{a+1,a+2,\dots,b\}$.

\subsection{Lattices} \label{sec:lattices} 
Suppose that $p$ has decomposition type $(\bfe,\bff)\in\N^g\times
\N^g$ in~$K$, in the sense defined in Subsection~\ref{subsec:1.1}.
Then $p$ decomposes in $K$ as $p \mathcal{O}_K = \p_1^{e_1} \cdots
\p_g^{e_g}$, where $\p_1, \dots, \p_g$ are distinct prime ideals
in~$\mcO_K$.  For each $i \in [g]$, let $k_i = \mathcal{O}_K / \p_i$
be the corresponding residue field.  Then $f_i = [k_i : \F_p]$.  We
define $C_i = \sum_{j = 1}^i e_j f_j$ for each $i \in [g]_0$, so that
$0 = C_0 < C_1 < \cdots < C_g = n$.

Let $R_p = \mathcal{O}_K \tensor_{\Z} \Z_p$.  This ring is a free
$\Z_p$-module of rank $n$.  It splits into a direct product $R_p =
R_p^{(1)} \times \cdots \times R_p^{(g)}$, where for each $i \in [g]$
the component $R_p^{(i)}$ is just the local ring $\mathcal{O}_{K,
  \p_i}$. For each $i \in [g]$, we choose a uniformizer $\pi_i \in
R_p^{(i)}$, an $\F_p$-basis $\{ \overline{\beta}_1^{(i)},
\dots, \overline{\beta}_{f_i}^{(i)} \}$ of~$k_i$, and a lift
$\beta_j^{(i)} \in R_p^{(i)}$ of $\overline{\beta}_j^{(i)} \in k_i$
for each~$j\in[f_i]$.  Then the set
$$ \mathcal{B}_i :=\left\{ \beta_j^{(i)} \pi_i^s \mid j \in [f_i],
s \in [e_i - 1]_0 \right\}
$$ is a basis of $R_p^{(i)}$ as a $\Z_p$-module;
see, for instance, the proof of \cite[Proposition
  II.6.8]{Neukirch/99}. The union of the bases $\mathcal{B}_i$, for $i
\in [g]$, constitutes a basis $\{ \alpha_1, \dots, \alpha_n \}$ of
$R_p$ as a $\Z_p$-module.  We index it as follows:
$$
\beta_j^{(i)} \pi_i^s = \alpha_{C_{i-1} + s f_i + j}.
$$ We define structure constants $c^{km}_u \in\Z_p$, for $k,m,u \in[n]$,
with respect to this basis, via
\begin{equation}\label{equ:structure.constants}
\alpha_k \alpha_m = \sum_{u=1}^n c^{km}_u \alpha_u.
\end{equation}
Note that $c^{km}_u=0$ unless there exists an $i \in [g]$ such that $k,m
\in\, ] C_{i-1}, C_i ]$.

Hence we obtain the following presentation of the $\Zp$-Lie ring $L_p = H(R_p)$:

\begin{equation*}
L_p = \left \langle x_1, \dots, x_n, y_1, \dots, y_n, z_1, \dots, z_n
\mid [x_k,y_m] = \sum_{u=1}^n c^{km}_u z_u, \text{ for }k,m \in[n]\right\rangle. 
\end{equation*}
Here it is understood that all unspecified Lie brackets vanish.  It is
clear that the center of this Lie ring, which is equal to the derived
subring, is spanned by $\{ z_1, \dots, z_n \}$.  Similarly, the
abelianization $\ol{L}_p = L_p / [L_p,L_p]$ is spanned by the images
of the elements $x_1, \dots, x_n, y_1, \dots, y_n$.  We abuse notation
and denote these elements of $\overline{L}_p$ by
$x_1,\dots,x_n,y_1,\dots,y_n$ as well.

Let $\overline{\Lambda} \leq \ol{L}_p$ be a sublattice of finite
index.  Then $\overline{\Lambda}$ is a free $\Z_p$-module of rank
$2n$.  Let $(b_1, \dots, b_{2n} )$ be an ordered $\Zp$-basis for
$\overline{\Lambda}$.  Observe that each $b_j$ can be expressed
uniquely in the form
\begin{equation}\label{equ:latt}
b_j = \sum_{k = 1}^n b_{2k - 1, j} x_k + \sum_{k = 1}^n b_{2k, j} y_k.
\end{equation}
for some $b_{1,j}, \dots, b_{2n, j} \in \Z_p$.  We set
$B(\overline{\Lambda}) = (b_{k,j}) \in \Mat_{2n}(\Z_p)$. Conversely,
the columns of any matrix $B\in\Mat_{2n}(\Zp)$ with $\det B \neq 0$
encode generators of a sublattice of $\ol{L}_p$ of finite index in
$\ol{L}_p$, by means of~\eqref{equ:latt}. The matrix
$B(\overline{\Lambda})$ depends on the choice of basis; indeed, two
matrices $B, B^\prime$ represent the same lattice if and only if there
exists some $A \in \GL_{2n}(\Z_p)$ such that $B^\prime = BA$.

If $F / \Q_p$ is a finite extension, we denote by $\val_F$ the
normalized valuation on $F$.  We simply write $\val$ instead of
$\val_{\Q_p}$. For each $i \in [g]$ we define the following two
parameters:
\begin{align}
  \epsi_i (\overline{\Lambda}) &= \min \left \{ \val(b_{k,j} ) \mid k
  \in\; ] 2 C_{i-1}, 2C_i],\, j \in [2n] \right
      \},\label{def:eps}\\ \delta_i (\overline{\Lambda}) & = \min
      \left \{ d \in[e_i-1]_0 \mid
        \right.\nonumber\\ &\quad \quad \left. \exists k \in\; ]
          2C_{i-1} + 2 d f_i, 2C_{i-1} + 2(d+1)f_i] , j \in [2n]
            :\;\val(b_{k,j}) = \epsi_i (\overline{\Lambda}) \right
            \}.\label{def:delta}
\end{align}
Informally, $\epsi_i (\overline{\Lambda})$ is the smallest valuation
of any element appearing on or between the $(2C_{i-1} + 1)$-st and
$(2C_i)$-th rows of the matrix $B(\overline{\Lambda})$.  If we split
this range of $2e_if_i$ rows into $e_i$ blocks of $2 f_i$ consecutive rows each,
then $\delta_i (\overline{\Lambda})$ is the largest number such that
the first $\delta_i (\overline{\Lambda})$ blocks contain no matrix
elements of minimal valuation $\epsi_i (\overline{\Lambda})$.  It is
easy to see that $\epsilon_i(\ol{\Lambda})$ and
$\delta_i(\ol{\Lambda})$ are independent of the choice of basis and so
are well-defined.

\begin{dfn}\label{def:ell}
For $j \in [n]$ we set
\begin{equation*}
\ell_j = \begin{cases} \epsi_i (\overline{\Lambda}) + 1, &\textrm{if }
  j \in\,]C_{i-1}, C_{i-1} + \delta_i(\overline{\Lambda})f_i],
      \\ \epsi_i (\overline{\Lambda}), &\textrm{if } j \in\,]C_{i-1} +
        \delta_i(\overline{\Lambda})f_i, C_i].
\end{cases}
\end{equation*}
and set $\ell (\overline{\Lambda}) = (\ell_1, \dots, \ell_n) \in
\N_0^n$.
\end{dfn}

Informally, the $n$-tuple $\ell(\overline{\Lambda})$ is a
concatenation of $g$ blocks of lengths $e_1 f_1, \dots, e_g f_g$.
Within each block, the components are all equal, except that for each
$i \in [g]$ the first $\delta_i (\overline{\Lambda}) f_i$ components
of the $i$-th block are incremented by $1$.  Thus
$\ell(\overline{\Lambda})$ just depends on the ramification type
$(\bfe, \bff)$ and the parameters $\varepsilon_i(\overline{\Lambda}),
\delta_i(\overline{\Lambda})$ for each $i \in [g]$.

\begin{lemma} \label{lem:gentype}
Let $\overline{\Lambda} \leq \ol{L}_p$ be a sublattice of finite
index, and let $\ell(\overline{\Lambda})$ be as in
Definition~\ref{def:ell}.  Then
$$L_p' / [\ol{\Lambda},L_p] \cong \prod_{j=1}^n \Z /{p^{\ell_j}} \Z.$$ 
\end{lemma}

\begin{proof}
  It is clear that 
\begin{equation} \label{equ:span}  
   [\overline{\Lambda}, L_p] = \mathrm{span}_{\Z_p} \left \{ [b_j,
     x_k], [b_j, y_k] \mid j \in [2n],\, k \in [n] \right\}.
  \end{equation}
  For each $i \in [g]$, we define the following sublattice of
  $[\overline{\Lambda}, L_p]$:
$$[\overline{\Lambda}, L_p]_i = \mathrm{span}_{\Zp}\left\{ [ b_j, x_k
  ], [b_j, y_k] \mid j \in [2n], k \in\,]C_{i-1}, C_i] \right\}.$$
By the observation following \eqref{equ:structure.constants}, $[\overline{\Lambda}, L_p] = \bigoplus_{i = 1}^g
[\overline{\Lambda}, L_p]_i$.  Moreover, if we set $L^\prime_p(i)$ to be
the $\Z_p$-submodule of $L_p'$ generated by $\{ z_{C_{i-1} + 1}, \dots,
z_{C_i} \}$, then it is clear that
$$ L_p^\prime / [\overline{\Lambda}, L_p] \simeq \prod_{i = 1}^g
L_p^\prime (i) / [\overline{\Lambda}, L_p]_i.$$ 

We have thus reduced to the case where $p$ is non-split in~$K$,
i.e.~$g=1$.  So suppose that $p \mathcal{O}_K = \p^e$ is non-split in
$K$ and write $\epsi, \delta$ for $\epsi_1 (\overline{\Lambda}),
\delta_1 (\overline{\Lambda})$ as in \eqref{def:eps},
\eqref{def:delta}.  Then $R_p$ is a local ring with residue field $k
\simeq \F_{p^f}$, where $ef = n$.  Let $\pi \in R_p$ be a uniformizer.
Let $F$ be the fraction field of $R_p$, and note that $(\val_F)_{|
  \Q_p} = e \cdot \val$.  As before, we choose a $\Z_p$-basis $(
\alpha_1, \dots, \alpha_n)$ of the form $\alpha_{sf + j} = \beta_j
\pi^s$, where $j \in [f]$ and $s \in [e-1]_0$, and the image in $k$ of
$\{ \beta_1, \dots, \beta_f \}$ is an $\F_p$-basis of $k$.

Let $\overline{\Lambda}$ be given by a matrix $B(\overline{\Lambda})
\in \Mat_{2n}(\Z_p)$ as above.  Then $\epsi$ is just the
minimal valuation attained by the entries of~$B
(\overline{\Lambda})$. To prove the lemma, it suffices to establish
the following claim.

\begin{claim*}
Let $(v_1, \dots, v_{2n}) \in \Z_p^{2n}$.  Set
\begin{equation*}
\epsi^\prime  =  \min \{ \val(v_{2k - 1}) \mid k \in [n] \}, \quad \epsi^{\prime \prime}  =  \min \{ \val(v_{2k}) \mid k \in [n] \}
\end{equation*}
and define
\begin{align*}
  \delta^\prime & =  \min \{d \in [e-1]_0 \mid \exists k \in\; ] df , (d+1)f] :\;\val(v_{2k - 1}) = \epsi^\prime \}, \\
  \delta^{\prime \prime} & = \min \{d \in [e-1]_0 \mid
  \exists k \in\; ] df, (d+1)f] :\;\val(v_{2k}) = \epsi^{\prime \prime}
  \}.
\end{align*}
Consider the element $v = \sum_{k = 1}^n (v_{2k - 1} x_k + v_{2k} y_k)
\in \overline{\Lambda}$.  Then
\begin{align} \label{equ: commspan}
 \mathrm{span}_{\Z_p} \{ [v, y_1], \dots, [v, y_n] \} = { }& \mathrm{span}_{\Z_p} \{
 p^{\epsi^\prime + 1} z_1, \dots, p^{\epsi^\prime + 1}
 z_{\delta^\prime f}, p^{\epsi^\prime} z_{\delta^\prime f + 1}, \dots,
 p^{\epsi^\prime} z_n \}, \\ \nonumber \mathrm{span}_{\Z_p} \{ [v, x_1],
 \dots, [v, x_n] \} = { }& \mathrm{span}_{\Z_p} \{ p^{\epsi^{\prime \prime} +
   1} z_1, \dots, p^{\epsi^{\prime \prime} + 1} z_{\delta^{\prime
     \prime} f}, p^{\epsi^{\prime \prime}} z_{\delta^{\prime \prime} f
   + 1}, \dots, p^{\epsi^{\prime \prime}} z_n \}.
\end{align}
\end{claim*}

Indeed, assuming the claim, it easily follows from \eqref{equ:span} that
$$[\overline{\Lambda}, L] = \mathrm{span}_{\Z_p} \{ p^{\epsi + 1} z_1,
\dots, p^{\epsi + 1} z_{\delta f}, p^\epsi z_{\delta f + 1}, \dots,
p^\epsi z_n \}.$$ Now we prove the claim.  We only consider the
statement involving $\epsi^\prime$ and $\delta^\prime$, since the
other half of the claim is dealt with analogously.  It is clear that
neither side of \eqref{equ: commspan} changes if we replace $v$ by
$v^\prime = \sum_{k = 1}^n v_{2k - 1} x_k$.  Moreover, replacing
$v^\prime$ with $p^{- \epsi^\prime} v^\prime$, we may assume without
loss of generality that $\epsi^\prime = 0$.

Now let $l\in[n]$ be the smallest number such that $\val(v_{2l - 1}) =
0$.  Then $l$ satisfies $l \in\; ] \delta'f, (\delta'+1)f]$ by the
    definition of $\delta^\prime$.  Observe that for each $m\in[n]$ we
    have, by~\eqref{equ:structure.constants},
\begin{equation} \label{equ:vprimeym}
[v^\prime, y_m] = \left[\sum_{k=1}^nv_{2k-1}x_k,y_m\right] = \sum_{u =
  1}^n \left( \sum_{k = 1}^n v_{2k - 1} c_u^{km} \right) z_u.
\end{equation}

It follows from our definition of the basis $(\alpha_1, \dots,
\alpha_n)$ that $\val_F (\alpha_k) = d$ if $k \in ] df, (d+1)f]$.  In
    particular, if $k > \delta^\prime f$, then $\val_F(\alpha_k) \geq
    \delta^\prime$ and hence $\val_F (\alpha_k \alpha_m) \geq
    \delta^\prime$ for all $m$.  Since the $\alpha_k$ are linearly
    independent over $\Z_p$, it follows that $\val_F (c_u^{km}
    \alpha_u) \geq \delta^\prime$ for all $u \in [n]$.  Thus, if $k >
    \delta^\prime f$ but $u \leq \delta^\prime f$, then we must have
    $\val_F (c_u^{km}) > 0$ and hence $\val (c_u^{km}) > 0$.  On the
    other hand, if $k \leq \delta^\prime f$ then $\val(v_{2k - 1}) >
    0$ by the definition of~$\delta^\prime$.  Therefore, if $u \leq
    \delta^\prime f$, then
    $\val\left(\sum_{k=1}^nv_{2k-1}c_u^{km}\right) \geq 1$.  It
    follows by \eqref{equ:vprimeym} that the left-hand side of
    \eqref{equ: commspan} is contained in the right-hand side.

Let $M =(M_{um})\in \Mat_n(\Z_p)$ be the matrix whose columns are
$[v^\prime, y_1], \dots, [v^\prime, y_n]$, with respect to the basis
$( z_1, \dots, z_n )$ of $L^\prime_p$.  Then $M_{um} = \sum_{k = 1}^n
v_{2k - 1} c_u^{km}$, and it follows from the definition of the
structure constants that $M$ is the matrix of the $\Z_p$-linear
operator
\begin{equation*}
R_p \to  R_p, \quad x \mapsto  \left( \sum_{k = 1}^n v_{2k - 1} \alpha_k \right) x
\end{equation*}
with respect to the basis $( \alpha_1, \dots, \alpha_n)$ of $R_p$.
Hence $\det M = \mathrm{N}_{F / \Q_p} \left( \sum_{k = 1}^n v_{2k - 1}
\alpha_k \right)$, where $\mathrm{N}_{F/\Q_p}$ denotes the norm
function. By the considerations in the previous paragraph we see that
all the entries in the first $\delta^\prime f$ rows of $M$ are
divisible by $p$.

Let $\Delta_{\delta^\prime f} \in \mathrm{GL}_n(\Q_p)$ be the diagonal
matrix such that the first $\delta^\prime f$ diagonal entries are
$p^{-1}$ and the remaining diagonal entries are $1$.  Let $M^\prime =
\Delta_{\delta^\prime f} M$.  Then $M^\prime \in M_n(\Z_p)$.  As
$\val_F \left( \sum_{k = 1}^n v_{2k - 1} \alpha_k \right) =
\delta^\prime$, it follows that $\val(\det M^\prime) = \val(\det M) -
\delta^\prime f = 0$.  Thus the matrix $M^\prime$ is invertible, and
the space spanned by its columns is just $L^\prime_p$.  It follows
that $p z_1, \dots, p z_{\delta^\prime f}$ are contained in the span
of the columns of $M$.  Hence the right-hand side of \eqref{equ:
  commspan} is contained in the left-hand side. This completes the
proof of the claim.
\end{proof}

\begin{dfn}\label{def:adm}
Let $(\bfe,\bff)\in\N^g\times\N^g$.  We say that an $n$-tuple $\ell =
(\ell_1, \dots, \ell_n) \in \N_0^n$ is {\emph {admissible for}}
  $(\bfe,\bff)$ if there exists a sublattice $\overline{\Lambda} \leq
  \ol{L}_p$ of finite index such that $\ell(\overline{\Lambda}) =
  \ell$.  This is equivalent to the condition that for, each $i \in
      [g]$, there exist $\delta_i \in [e_i - 1]_0$ such that
\begin{equation}\label{equ:C}
\ell_{C_{i-1} + 1} = \cdots = \ell_{C_{i-1} + \delta_i f_i} =
\ell_{C_{i-1} + \delta_i f_i + 1} + 1 = \cdots = \ell_{C_i} + 1.
\end{equation} 
We denote the set of admissible $n$-tuples by ${\mathrm{Adm}}_{\bfe,
  \bff} \subseteq \N_0^n$.
\end{dfn} 

 We sometimes make use of the fact that an admissible $n$-tuple $\ell$
 determines, and is determined by, the pair of $g$-tuples
 $((\ell_{C_1}, \dots, \ell_{C_g}), (\delta_1, \dots, \delta_g))$
 in~\eqref{equ:C}.  Note that $\Adm_{\mathbf{1}, \mathbf{1}}
 =\N_0^n$. The opposite extreme occurs for $(\bfe,\bff)=((1),(n))$,
 where $\Adm_{(1),(n)} = \bfo\N_0$ consists of $n$-tuples all of whose
 components are equal.

Recall that, for $d\in\N$,
\begin{equation}\label{equ:ab}
\zeta^\vartriangleleft_{\Z_p^d} (s) = \prod_{i=0}^{d-1} \zeta_p(s-i) =
\frac{1}{\prod_{i=0}^{d-1}(1-p^{i-s})},
\end{equation}
where $\zeta_p(s) = (1-p^{-s})^{-1}$ is the local Riemann zeta
function; cf., for instance, \cite[Proposition~1.1]{GSS/88}.  Since
$\ol{L}_p$ is a free abelian Lie ring of rank $2n$ over $\Z_p$, we
have that $\zeta^\vartriangleleft_{\ol{L}_p} (s) =
\zeta^\vartriangleleft_{\Z_p^{2n}}(s)$.

\begin{lemma}\label{lem:l} 
Let $p$ be a prime of decomposition type $(\bfe, \bff)$ in~$K$.  Given
an $n$-tuple $\ell=(\ell_1,\dots,\ell_n)\in \Adm_{\bfe, \bff}$, we
have
$$\sum_{\ol{\Lambda}\leq_f \ol{L}_p,\,\ell(\ol{\Lambda}) =
  \ell}|\ol{L}_p:\ol{\Lambda}|^{-s}= \frac{\left( \prod_{i = 1}^g (1 -
  t^{2f_i}) \right) t^{2 \sum_{i=1}^n
    \ell_i}}{\prod_{i=0}^{2n-1}(1-p^{i}t)} = \left( \prod_{i = 1}^g
(1 - t^{2f_i}) \right) \zeta^\vartriangleleft_{\Zp^{2n}}(s) t^{2 \sum_{i=1}^n
  \ell_i}.$$
\end{lemma}

\begin{proof}
Denote the leftmost object in the equality above by $\Sigma_\ell$.  We
first prove that
\begin{equation} \label{equ:sigmaell}
 \Sigma_{\ell} = t^{2 \sum_{i=1}^n \ell_i}\Sigma_\bfz,
\end{equation}
where $\bfz$ denotes the zero vector $(0,\dots,0)\in\N_0^n$.  Indeed, there is a bijection
$\psi$ from matrices representing finite-index sublattices with
$\ell(\overline{\Lambda}) = \bfz$ to those representing finite-index
sublattices with $\ell(\overline{\Lambda}) = \ell$ given as follows.
Given a matrix $B\in\Mat_{2n}(\Zp)$, we define $\psi(B) = D P B$,
where $P$ is the permutation matrix representing the permutation
$$ \prod_{i = 1}^g (2 C_{i-1} + 1 \mbox{ } 2 C_{i-1} + 2 \mbox{ }
\cdots \mbox{ } 2 C_i )^{2 \delta_i f_i} \in S_{2n},$$
 and $D$ is the diagonal
matrix $\mathrm{diag}(d_1, \dots, d_{2n})$ whose entries are
$$ d_k = \begin{cases} p^{\ell_{C_i} + 1}, &\textrm{ if } k \in\, ] 2 C_{i-1}
  , 2 (C_{i-1} +  \delta_i f_i)], \\ p^{\ell_{C_i}}, &\textrm{ if }
 k \in\, ]2(C_{i-1} + \delta_i f_i), 2C_i].
\end{cases}
$$
Informally, within each block of $2 e_i f_i$ rows of $B$, we multiply
everything by $p^{\ell_{C_i}}$, then we cyclically move each row down
$2 \delta_i f_i$ places and multiply the top $2 \delta_i f_i$ rows of
the resulting matrix by $p$.  It is easy to see that this yields a
bijection as claimed, and, since left multiplication commutes with
right multiplication, it obviously induces a bijection between
lattices with $\ell(\overline{\Lambda}) = \bfz$ and those with
$\ell(\overline{\Lambda}) = \ell$; we also denote this bijection by
$\psi$.  Moreover, we observe that if the matrix $B$ represents a
finite-index sublattice $\overline{\Lambda}\leq \ol{L}_p$, then $|
\ol{L}_p : \overline{\Lambda} | = p^{\val (\det B)}$.  Since $\det
\psi(B) = p^{2 \sum_{i=1}^n\ell_i} \det B$, we conclude that indeed
$$ \Sigma_\ell = \sum_{\overline{\Lambda} \leq_f \ol{L}_p,\, \ell(\overline{\Lambda}) = \bfz} | \ol{L}_p :
\psi(\overline{\Lambda}) |^{-s} = t^{2\sum_{i=1}^n\ell_i}
\sum_{\overline{\Lambda} \leq_f \ol{L}_p ,\, \ell(\overline{\Lambda})
  = \bfz} | \ol{L}_p : \overline{\Lambda} |^{-s} =
t^{2\sum_{i=1}^n\ell_i} \Sigma_{\bfz}.$$

We observe that
$$ \sum_{\overline{\Lambda} \leq_f \ol{L}_p \atop
  \ell(\overline{\Lambda}) \in \Adm_{\bfe, \bff}} | \ol{L}_p :
\overline{\Lambda} |^{-s} = \zeta^\vartriangleleft_{\ol{L}_p}(s)$$ by
definition, since the sum runs over all finite-index sublattices
of~$\ol{L}_p$; since $\ol{L}_p$ is abelian, they are all ideals.
Using the characterization of $\ell \in \Adm_{\bfe, \bff}$ via the
$\ell_{C_i}$ and $\delta_i$ in~\eqref{equ:C}, we see that
\begin{align*}
 \frac{\zeta^\vartriangleleft_{\ol{L}_p}(s)}{\Sigma_{\bfz}} & = \sum_{\ell
   \in \Adm_{\bfe, \bff}} t^{2\sum_{i=1}^n\ell_i} \\ & = \sum_{(\ell_{C_1},
   \dots, \ell_{C_g}) \in \N_0^g} \sum_{\delta_1 = 0}^{e_1 -1} \cdots
 \sum_{\delta_g = 0}^{e_g - 1} t^{2\sum_{i=1}^gf_i(e_i \ell_{C_i}+\delta_i)}  \\ & = \prod_{i =
   1}^g \left( \sum_{\ell_{C_i} = 0}^{\infty} (t^{2e_i
   f_i})^{\ell_{C_i}} \right) (1 + t^{2 f_i} + (t^{2f_i})^2 + \cdots +
 (t^{2 f_i})^{e_i - 1}) \\ & = \prod_{i = 1}^g \frac{(1 + t^{2 f_i} +
   (t^{2f_i})^2 + \cdots + (t^{2 f_i})^{e_i - 1}) }{1 - (t^{2
     f_i})^{e_i} }= \prod_{i = 1}^g \frac{1}{1 - t^{2 f_i}}.
 \end{align*}
Therefore $\Sigma_{\bfz} = \left( \prod_{i = 1}^g (1 - t^{2
  f_i}) \right) \zeta^\vartriangleleft_{\ol{L}_p}(s)$.  Together
with~\eqref{equ:sigmaell}, this establishes the lemma.
\end{proof}

\subsection{Igusa functions}
Recall that, for a variable $Y$ and integers
$a,b\in\N_0$ with $a\geq b$, the \emph{Gaussian polynomial} (or
\emph{Gaussian binomial coefficient}) is defined to be
$$\binom{a}{b}_Y = \frac{\prod_{i=a-b+1}^a (1-Y^i)}{\prod_{i=1}^b
  (1-Y^i)}\in \Z[Y].$$
Given an integer $n\in\N$ and a subset $I \subseteq [n-1]$ whose
elements are $i_1 < i_2 < \cdots < i_m$, the associated \emph{Gaussian
  multinomial} is defined as
$$ \binom{n}{I}_Y = \binom{n}{i_{m}}_Y \binom{i_{m}}{i_{m-1}}_Y \cdots
\binom{i_2}{i_1}_Y \in \Z [Y].$$

\begin{dfn}\label{def:igusa}
Let $\wo\in\N$. Given variables $Y$ and $\bfX=(X_1,\dots,X_\wo)$, we
set
\begin{align*}
I_\wo(Y; \bfX)& =
\frac{1}{1-X_\wo}\sum_{I\subseteq[\wo -1]} \binom{\wo}{I}_{Y} \prod_{i\in
  I}\frac{X_i}{1-X_i}\;\in\Q(Y,X_1,\dots,X_\wo),\\ 
 I_\wo^\circ (Y; \bfX) & =
\frac{X_\wo}{1-X_\wo}\sum_{I\subseteq[\wo-1]} \binom{\wo}{I}_Y \prod_{i\in
  I}\frac{X_i}{1-X_i}\;\in\Q(Y,X_1,\dots,X_\wo).
\end{align*}
\end{dfn}
As mentioned in the introduction, an important feature of these
functions for us is that they satisfy a functional equation upon
inversion of the variables; see Proposition~\ref{pro:funeq.igusa}.

\begin{rem} \label{rem:igusa.fnct}
 The function $I_\wo$ is -- up to the factor $\frac{1}{1-X_\wo}$ --
 equal to the function $F_\wo$ defined in~\cite[Theorem 4]{Voll/05}.
 We consider it more natural to include the factor in the definition here.
\end{rem}

\begin{exm} 
\begin{equation*}
  I_1(Y;X_1) = \frac{1}{1-X_1}, \quad I_2(Y;X_1,X_2) =
  \frac{1}{1-X_2}\left(1 + (1 + Y) \frac{X_1}{1-X_1}\right).
\end{equation*} 
\end{exm}

\subsection{Weak orderings, flag complexes, and generalized Igusa
  functions}\label{subsec:weak.order} When dealing with unramified
primes which are not totally split, we will need to work with a larger
class of rational functions than the Igusa functions of
Definition~\ref{def:igusa}. These variant Igusa functions, which
generalize the functions $I_\wo (1; \mathbf{X})$ by
Lemma~\ref{lem:ig.genig}, will be defined in the terminology of weak
orderings and flag complexes.  We now explain these notions.

Let $\wo\in\N$. The symmetric group $S_\wo$ of degree $\wo$ is a
Coxeter group, with Coxeter generating set $\mcS =
\{s_1,\dots,s_{\wo-1}\}$, where $s_i$ corresponds to the transposition
$(i\,\, i+1)$ in the standard permutation representation
of~$S_\wo$. The (\emph{Coxeter}) \emph{length} $\len(\sigma)$ of an
element $\sigma\in S_\wo$ is the length of a shortest word
representing $\sigma$ as a product of elements of~$\mcS$. Given
$\sigma\in S_\wo$, we define its (\emph{right}) \emph{descent set}
$$\Des(\sigma) = \{ i\in [\wo -1] \mid \len(\sigma s_i) < \len(\sigma)
\}.$$ It is well known that $\Des(\sigma) = \{ i\in [\wo -1] \mid
\sigma(i) > \sigma(i+1) \}$; see, for instance, \cite[Proposition
  1.5.3]{BjoernerBrenti/05}. Given a set $A$, we denote by $2^A$ the
set of all subsets of $A$.

\begin{dfn}
  A \emph{weak ordering on $\wo$} is a pair $(\sigma,J)\in S_\wo
  \times 2^{[\wo-1]}$ such that $\Des(\sigma)\subseteq J$.  We
  set $$\WO_\wo = \{(\sigma,J)\in S_\wo\times 2^{[\wo-1]} \mid
  \Des(\sigma)\subseteq J\}.$$
\end{dfn}

Informally, a weak ordering is a possible outcome of a race among
$\wo$ contestants, if ties are permitted. Given $(\sigma,J)\in
\WO_\wo$, where the elements of $J$ are $j_1 < \cdots < j_\ell$, the contestants
$\sigma(1),\dots,\sigma(j_1)$ share the first place,
$\sigma(j_1+1),\dots,\sigma(j_2)$ share the second place, etc.

Weak orderings may be also interpreted in terms of face complexes.
Consider $\Gamma_\wo$, the first barycentric subdivision of the
boundary $D_\wo$ of the $(\wo-1)$-simplex on $\wo$ vertices. Let
$P_\wo$ be its face complex. Thus $P_\wo = \mathcal{F}(\Gamma_h)$ and
$\Gamma_\wo=\Gamma(P_h)$ in the notation of
\cite[Section~1]{Stanley/91}. We may interpret $P_\wo$ as the poset of
chains of nontrivial and proper subsets of $[\wo]$. The empty chain
plays the role of the initial object $\widehat{0}$.  A general element
$y\in P_\wo$ has the form
$$y = (\mcI_1 \subsetneq \mcI_2 \subsetneq \dots \subsetneq \mcI_\ell),$$
where $\mcI_i\subsetneq[\wo]$ for each $i\in[\ell]$. The map
\begin{equation} \label{equ:poset.isom}
\phi:\WO_\wo \rarr P_\wo,\quad
(\sigma, J) \mapsto (\{\sigma(1),\dots,\sigma(j)\})_{j\in J}
\end{equation}
is a poset isomorphism.

Next we define a class of functions, partially generalizing the Igusa
functions introduced in Definition~\ref{def:igusa}. Given
$\mcI\subseteq [\wo]$ and $y\in P_h$, we say $\mcI\in y= (\mcI_1
\subsetneq \mcI_2 \subsetneq \dots \subsetneq \mcI_\ell)$ if
$\mcI=\mcI_i$ for some~$i\in[\ell]$.

\begin{dfn}\label{dfn:gen.igusa}
  Let ${\bf X} = (X_\mcI)_{\mcI \in
    2^{[\wo]}\setminus\{\varnothing\}}$ be a collection of variables
  parameterized by the non-empty subsets of $[h]$.  Define
$$\Iwo_\wo({\bf X}) = \frac{1}{1-X_{[\wo]}}\sum_{y\in P_\wo}
  \prod_{\mcI\in y} \frac{X_{\mcI}}{1-X_{\mcI}}.$$
\end{dfn}

\begin{rem}\label{rem:face.ring}
 Alternatively, one may view $\Iwo_\wo({\bf X})$ as a fine Hilbert
 series of a face ring. Indeed, let $k$ be any field, $\Delta_\wo$ the
 first barycentric subdivision of the $(\wo-1)$-simplex, and
 $k[\Delta_\wo]$ the associated face (Stanley-Reisner) ring;
 cf.\ \cite[Chapter~II, Section~1]{Stanley/96}. One verifies easily
 that $\Iwo_\wo(\bfX)$ is the fine Hilbert series of $k[\Delta_\wo]$.
\end{rem}

\begin{lem} \label{lem:ig.genig}
Given variables ${\bf X} = (X_\mcI)_{\mcI \in
  2^{[\wo]}\setminus\{\varnothing\}}$ and $\mathbf{Z}=(Z_1,\dots,Z_\wo)$,
the substitutions
$$X_{\mcI} \rarr Z_{|\mcI|}, \quad \mcI\subseteq [\wo]$$ map
$\Iwo_\wo(\bfX)$ to $I_\wo(1;Z_1,\dots,Z_\wo)$.
\end{lem}

\begin{proof}
It is well known (see, for instance, \cite[Proposition
  1.4.1]{Stanley/12}) that, given $J \subseteq [\wo -1]$,
\begin{equation}  \label{equ:binom.descent}
\#\{\sigma\in S_\wo \mid \Des(\sigma) \subseteq J\} = \binom{\wo}{J}.
\end{equation}
Since the map of~\eqref{equ:poset.isom} is a poset isomorphism, this
implies that
\begin{align*}
\Iwo_\wo((Z_{|\mcI|})_{\mcI\subseteq
  2^{[\wo]}\setminus\{\varnothing\}})&=\frac{1}{1-Z_\wo}
\sum_{I\subseteq[\wo-1]}\#\{\sigma\in S_\wo \mid\Des(\sigma)\subseteq
I\} \prod_{i\in I}\frac{Z_i}{1-Z_i}\\ &=\frac{1}{1-Z_\wo}
\sum_{I\subseteq[\wo-1]}\binom{\wo}{I} \prod_{i\in
  I}\frac{Z_i}{1-Z_i}=I_\wo(1;Z_1,\dots,Z_\wo),
\end{align*}
as claimed.
\end{proof}

\begin{exm}
  Let $\wo=3$. For a variable $Y$, denote $\gp{Y} = \frac{Y}{1-Y}$.
  We have
\begin{multline*}
  \Iwo_3(X_1,X_2,X_3,X_{12},X_{13},X_{23},X_{123}) =\\
  \frac{1}{1-X_{123}}\left(1 + \gp{X_1} + \gp{X_2} +
    \gp{X_3}+\gp{X_{12}} + \gp{X_{13}} + \gp{X_{23}} \right.\\
  \left. +\gp{X_1}\gp{X_{12}}+ \gp{X_1}\gp{X_{12}}+
    \gp{X_2}\gp{X_{12}}\right.\\\left.+\gp{X_2}\gp{X_{13}}+
    \gp{X_3}\gp{X_{13}}+ \gp{X_3}\gp{X_{23}}\right),
\end{multline*}
whereas
\begin{multline*}
  I_3(1;Z_1,Z_2,Z_3) = \\ \frac{1}{1-Z_3} \left( 1 +
  \binom{3}{1}\frac{Z_1}{1-Z_1} + \binom{3}{2}\frac{Z_2}{1-Z_2} +
  \binom{3}{\{1,2\}} \frac{Z_1Z_2} {(1-Z_1)(1-Z_2)} \right).
\end{multline*}
\end{exm}

\begin{rem} \label{rem:coxeter}
 We note a consequence of \eqref{equ:binom.descent} for future use.
 Let $w_0 \in S_\wo$ be the unique element of highest Coxeter length;
 it corresponds to the permutation $i \mapsto \wo + 1 - i$ and has
 order two.  It is easy to check that for any $\sigma \in S_\wo$, we
 have $\Des(w_0 \sigma w_0) = \wo - \Des(\sigma)$.  Here for any
 subset $J \subseteq [\wo - 1]$ we denote $\wo - J = \{ \wo - j \mid j
 \in J \}$.  Since conjugation by $w_0$ is an automorphism of $S_\wo$,
 it is immediate from \eqref{equ:binom.descent} that
$$ \binom{\wo}{\wo-J} = \binom{\wo}{J}.$$

More generally, for a variable $Y$, by means of the identity
\cite[Proposition 1.7.1]{Stanley/12}
$$ \binom{\wo}{J}_Y = \sum_{\sigma \in S_\wo \atop \Des(\sigma)
  \subseteq J} Y^{\len(\sigma)}$$ and the observation that $\len(w_0
\sigma w_0) = \len(\sigma)$ for all $\sigma \in S_\wo$, we obtain
that
$$ \binom{\wo}{\wo-J}_Y = \binom{\wo}{J}_Y.$$
\end{rem}

\subsection{Pairs of partitions and Dyck words} \label{sec:dyckwords}

Let $\mu = (\mu_1,\dots,\mu_n)$ and $\lambda =
(\lambda_1,\dots,\lambda_n)$ be partitions of $n$ non-negative parts
such that $\lambda$ dominates $\mu$, that is $\mu_1\geq \dots \geq
\mu_n \geq 0$ and $\lambda_1\geq \cdots \geq \lambda_n \geq 0$ and
$\mu_i \leq \lambda_i$ for all $i \in [n]$.  This last condition is
abbreviated by $\mu \leq \lambda$.  There are uniquely determined
integers $r\in\N_0$ and $M_i, L_i\in\N$ ($i=1,\dots,r-1$), such that
\begin{multline} \label{mult:limi}
 \lambda_1 \geq \cdots \geq \lambda_{L_1} \geq \mu_1 \geq \cdots \geq
 \mu_{M_1} > \lambda_{L_1 + 1} \geq \cdots \geq \lambda_{L_2} \geq
 \mu_{M_1 + 1} \geq \cdots \geq \mu_{M_2} > \cdots \\ >
 \lambda_{L_{r-1}+1} \geq \cdots \geq \lambda_n \geq \mu_{M_{r-1} + 1}
 \geq \cdots \geq \mu_n.
\end{multline}
Define $L_r = M_r = n$ and $L_0 = M_0 = 0$, and observe that the
condition $\mu \leq \lambda$ is equivalent to the condition that $L_i
\geq M_i$ for all $i\in[r]$.

A {\emph{Dyck word of length $2n$}} is a word $w$ in the
letters $\bfz$ and $\bfo$, such that $\bfz$ and $\bfo$ each occur $n$
times in $w$ and no initial segment of $w$ contains more ones than
zeroes.  Equivalently, a Dyck word is a well-parsed sequence of $n$
open parentheses and $n$ closed parentheses.  We denote the set of
Dyck words of length $2n$ by $\mathcal{D}_{2n}$ and note that the
cardinality of $\mathcal{D}_{2n}$ is the $n$-th Catalan number
$\mathrm{Cat}_n=\frac{1}{n+1}\binom{2n}{n}$.  For example, $\mcD_{6} =
\{\bfz\bfz\bfz\bfo\bfo\bfo,\bfz\bfz\bfo\bfz\bfo\bfo,\bfz\bfz\bfo\bfo\bfz\bfo,
\bfz\bfo\bfz\bfz\bfo\bfo,\bfz\bfo\bfz\bfo\bfz\bfo\}$.
See~\cite[Example~6.6.6]{Stanley/99} for more details about Dyck
words.

Given a pair of partitions $\mu \leq
\lambda$ of at most $n$ parts as above, define the Dyck word
\begin{equation*}
w(\mu, \lambda) = \bfz^{L_1} \bfo^{M_1} \bfz^{L_2 - L_1} \bfo^{M_2 -
  M_1} \cdots \bfz^{n - L_{r-1}} \bfo^{n - M_{r-1}}
\in\mathcal{D}_{2n}.
\end{equation*}
In other words, the word $w(\mu, \lambda)$ consists of $L_1$ zeroes
followed by $M_1$ ones, followed by $L_2 - L_1$ zeroes, etc.  The
condition $\mu \leq \lambda$ ensures that $w(\mu, \lambda)$ is indeed
a Dyck word.  Observe that the Dyck word $w(\mu,
\lambda)\in\mathcal{D}_{2n}$ determines, and is determined by, the
collection of integers $\{L_i, M_i \}_{i\in[r]}$ from
\eqref{mult:limi}.  It is useful for us to have notation for the
successive differences of the parts of $\lambda$ and $\mu$.  We set,
for $j\in[n]$,
\begin{equation} \label{equ:defrj}
r_j = \begin{cases} \mu_j - \mu_{j + 1}, &\textrm{ if } j \not\in \{
  M_1, \dots, M_r \}, \\ \mu_{M_i} - \lambda_{L_i + 1}, &\textrm{ if }
  j = M_i.
\end{cases}
\end{equation}
where we define $\lambda_{n + 1} = 0$.  Similarly, we recall that $M_0 = 0$
and put, for $j\in[n]$,
\begin{equation} \label{equ:defsj}
s_j = \begin{cases} \lambda_j - \lambda_{j + 1}, &\textrm{ if } j
  \not\in \{L_1, \dots, L_r \}, \\ \lambda_{L_i} - \mu_{M_{i - 1} + 1},
  &\textrm{ if } j = L_i.
\end{cases}
\end{equation}

Note that $r_j>0$ for $j\in\{M_1,\dots,M_{r-1} \}$ and observe that
$\mu_{M_i} > \mu_{M_i + 1}$ and $\lambda_{L_i} > \lambda_{L_i + 1}$
for each $i\in[r-1]$.  Finally, for each $i \in [r]$ we define
\begin{align} \label{equ:jimu}
J^\mu_i &= \{ j \in [M_i - M_{i-1} - 1] \mid \mu_{M_i - j} >
\mu_{M_{i} - j + 1} \},\\ \nonumber
J^\lambda_i &= \{ j \in [L_i - L_{i - 1} -
  1] \mid \lambda_{L_i - j} > \lambda_{L_{i} - j + 1} \}.
\end{align}

Given a partition $\lambda$, we set, for $i\in\N$,
$$\lambda_i':= \# \{j\in\N \mid \lambda_j \geq i \}.$$ The partition
$\lambda' = (\lambda_1',\lambda_2',\dots)$ is called the \emph{dual
  partition} of $\lambda$. Observe that, if $\lambda$ has at most $n$
parts, then the parts of $\lambda'$ are bounded by $n$. In this case
we write $\mcJ(\lambda) = \{ j \in [n-1] \mid \lambda_j > \lambda_{j +
  1} \}$ for the set of positive parts of~$\lambda'$.

Given $\ell \in \N_0^n$ we let $\lambda(\ell)$ be the partition
obtained by arranging the entries of $\ell$ in non-ascending order. We
let $\beta(\lambda)$ be the number of $n$-tuples $\ell \in \N_0^n$
such that $\lambda(\ell) = \lambda$.

\begin{lemma} \label{lem:lambda} Let
  $\mathcal{L}=\{L_1,\dots,L_{r-1}\} \subseteq [n - 1]$ be as above.
  Then
  \begin{equation*} 
  \beta(\lambda) =
    \binom{n}{\mcJ(\lambda)} = \binom{n}{\mathcal{L}} \prod_{i = 1}^r
    \binom{L_i - L_{i-1}}{J^\lambda_i}.
\end{equation*}
\end{lemma}
\begin{proof} The first equation is clear. The second follows from the
  observation that
\begin{equation*}
 \mcJ(\lambda) = \mathcal{L} \cup \bigcup_{i = 1}^r \{ L_i - j \mid j \in
J_i^\lambda \}. \qedhere
\end{equation*}
 \end{proof}

\subsection{Subgroups of abelian $p$-groups}
In order to evaluate sums like \eqref{equ:intro.dpt}, we need to
understand, given a pair of partitions $\mu \leq \lambda$, the numbers
$\alpha(\lambda, \mu;p)$ of abelian $p$-groups of type $\mu$ contained
in a fixed abelian $p$-group of type $\lambda$.  We recall here an
explicit formula for these numbers, attributed to Birkhoff
in~\cite{Butler/87}.

\begin{pro}[Birkhoff] \label{pro:birkhoff} 
Let $\mu \leq \lambda$ be partitions, with dual partitions
$\mu^\prime \leq \lambda^\prime$.  Then
\begin{equation*}
\alpha(\lambda, \mu;p) = \prod_{k \geq 1} p^{\mu_k^\prime
  (\lambda_k^\prime - \mu_k^\prime)} \binom{\lambda_k^\prime -
  \mu_{k+1}^\prime}{\lambda_k^\prime - \mu_k^\prime}_{p^{-1}}.
\end{equation*}
\end{pro}

\begin{lemma} \label{lem:mu.split} Let $\mu \leq \lambda$ be
  partitions, and let $r\in\N$ and $\{L_i,M_i\}_{i\in[r]}$ be the
  parameters associated to them in~\eqref{mult:limi}.  Then, for
  $i\in[r-1]$,
\begin{multline} \label{equ:birkhoff.mupiece}
\prod_{k = \lambda_{L_i + 1} + 1}^{\mu_{M_{i-1} + 1}} p^{\mu_k^\prime
  (\lambda_k^\prime - \mu_k^\prime)} \binom{\lambda_k^\prime -
  \mu_{k+1}^\prime}{\lambda_k^\prime - \mu_k^\prime}_{p^{-1}} =
\\ \prod_{j = 1}^{M_i - M_{i - 1}} p^{(M_{i - 1} + j)(L_i - M_{i - 1}
  - j)r_{M_{i-1} + j}} \binom{M_i - M_{i - 1}}{J^\mu_i}_{p^{-1}} \cdot
\binom{L_i - M_{i-1}}{L_i - M_i}_{p^{-1}}.
\end{multline}
\end{lemma}
\begin{proof}
Observe that all the indices $k$ appearing in the product on the left
hand side satisfy $\lambda_{L_{i}+1} < k \leq \mu_{M_{i-1}+1} \leq
\lambda_{L_i}$, and hence $\lambda^\prime_k = L_i$.  Moreover, it is
easy to see that $\mu_k^\prime = M_{i - 1} + j$ when $\mu_{M_{i - 1} +
  j + 1} < k \leq \mu_{M_{i-1} + j}$ holds; observe that it may be the
case for some $j$ that no index $k$ satisfies this condition.  As a
result, we see that for each $j \in [M_i - M_{i-1}]$, there are
exactly $r_{M_{i-1} + j}$ elements $k$ of the segment $]\lambda_{L_{i}
    + 1}, \mu_{M_{i-1} + 1}]$ for which $\mu^\prime_k = M_{i - 1} +
    j$.

Observe that the Gaussian binomial coefficients on the left-hand side
of \eqref{equ:birkhoff.mupiece} differ from $1$ only when
$\mu_k^\prime \neq \mu_{k+1}^\prime$, namely when $k$ is a part of the
partition $\mu$, i.e.~there exists an $i$ such that $\mu_i = k$.  It
follows that if $J^\mu_i = \{j_{i,1} , \cdots , j_{i, \gamma_i}\}$,
with $j_{i,1} < \cdots < j_{i, \gamma_i}$, then
\begin{multline} \label{binom.comps.split}
\prod_{k = \lambda_{L_i + 1} + 1}^{\mu_{M_{i-1}}}  \binom{\lambda_k^\prime - \mu_{k+1}^\prime}{\lambda_k^\prime - \mu_k^\prime}_{p^{-1}} = \\ \binom{L_i - M_i + j_{i,\gamma_i}}{L_i - M_i}_{p^{-1}} \cdot
\prod_{m = 1}^{\gamma_i - 1} \binom{L_i - M_i + j_{i,m + 1}}{L_i - M_i + j_{i,m}}_{p^{-1}} \cdot \binom{L_i - M_{i-1}}{L_i - M_i + j_{i,\gamma_i}}_{p^{-1}}.
\end{multline}
We make use of the well-known identity
\begin{equation*}
\binom{\alpha}{\beta}_Y = \frac{1 - Y^\alpha}{1 - Y^{\alpha - \beta}} \binom{\alpha - 1}{\beta}_Y
\end{equation*}
for Gaussian binomial coefficients.  Applying it inductively, we see
that for all $m \in [\gamma_i - 1]$,
\begin{equation*}
\binom{L_i - M_i + j_{i,m + 1}}{L_i - M_i + j_{i,m}}_{p^{-1}} = \binom{j_{i,m+1}}{j_{i,m}}_{p^{-1}} \frac{\binom{L_i - M_i + j_{i,m+1}}{L_i - M_i}_{p^{-1}}}{\binom{L_i - M_i + j_{i,m}}{L_i - M_i}_{p^{-1}}}.
\end{equation*}
Hence the right-hand side of \eqref{binom.comps.split} is equal to 
\begin{equation*}
\binom{M_i - M_{i - 1}}{J^\mu_i}_{p^{-1}} \cdot \binom{L_i - M_{i-1}}{L_i - M_i}_{p^{-1}}
\end{equation*}
and our claim follows.
\end{proof}

\begin{lemma} \label{lem:lambda.p.split}
Let $\mu \leq \lambda$ be partitions, with dual partitions
$\mu'\leq\lambda'$.  Then, for $i\in[r-1]$,
\begin{equation*}
\prod_{k = \mu_{M_{i-1} + 1} + 1}^{\lambda_{L_{i-1} + 1}} p^{\mu_k^\prime (\lambda_k^\prime - \mu_k^\prime)} \binom{\lambda_k^\prime - \mu_{k+1}^\prime}{\lambda_k^\prime - \mu_k^\prime}_{p^{-1}} = \prod_{j = 1}^{L_i - L_{i-1}} p^{M_{i-1} (L_{i-1} - M_{i-1} + j) s_{L_{i-1}+j}}.
\end{equation*}
\end{lemma}
\begin{proof}
Note that the product on the left-hand side may be empty; this happens
in the case $\lambda_{L_{i-1} + 1} = \cdots = \lambda_{L_i} =
\mu_{M_{i-1} + 1}$.  All of the Gaussian binomial coefficients on the
left-hand side are equal to $1$, since the interval $]\mu_{M_{i-1} +
    1}, \lambda_{L_{i-1} + 1}]$ contains no parts of the partition
    $\mu$.  Moreover, we observe that $\mu_k^\prime = M_{i-1}$ for all
    $k$ in this interval.  Finally, observe that for $j \in [L_i -
      L_{i-1}]$ we have $\lambda_k^\prime = L_{i-1} + j$ when
    $\lambda_{L_{i-1}+j+1} < k \leq \lambda_{L_{i-1} + j}$ holds.  The
    claim follows as in the proof of the previous lemma.
\end{proof}

\subsection{Rewriting the zeta function}\label{subsec:rewriting}

Let $p$ be a prime of decomposition type $(\bfe,\bff)$ in~$K$. We put
our work so far to use to rewrite the zeta function $\zideal_{L_p}(s)$.

\begin{dfn}
Given $(\bfe,\bff)\in\N^g\times\N^g$ with $\sum_{i = 1}^g e_i f_i =
n$, we set
$$D^{\bfe,\bff}(p,t) =
\sum_{\ell\in\Adm_{\bfe,\bff}}t^{2\sum_{i=1}^n\ell_i}\sum_{\mu\leq\lambda(\ell)}\alpha(\lambda(\ell),\mu;p)(p^{2n}t)^{\sum_{i=1}^n
  \mu_i}.$$
\end{dfn}

\begin{lem}\label{lem:rewrite}
Let $p$ be a prime of decomposition type $(\bfe, \bff)$ in $K$.  Then
$$\zideal_{L_p}(s) =\left( \prod_{i=1}^g (1 - t^{2 f_i}) \right)
  \,\zeta^\vartriangleleft_{\Z_p^{2n}}(s) \,D^{\bfe, \bff} (p,t).$$
\end{lem}

\begin{proof} Using \eqref{equ:intro.basic.eq} and Lemma~\ref{lem:l}, we obtain
\begin{align*}
\zideal_{L_p}(s) &= \sum_{\ol{\Lambda} \leq_f \ol{L}_p} \vert \ol{L}_p
: \ol{\Lambda} \vert^{-s} \sum_{[\ol{\Lambda},L_p] \leq M \leq L_p'}
\vert L'_p : M \vert ^{2n-s}\\ &=
\sum_{\ell\in\Adm_{\bfe,\bff}}\sum_{\mu\leq\lambda(\ell)}\alpha(\lambda(\ell),\mu;p)\left(p^{2n}t\right)^{\sum_{i=1}^n\mu_i}\sum_{\ol{\Lambda}\leq_f
  \ol{L}_p,\;\ell(\ol{\Lambda})=\ell} \vert \ol{L}_p : \ol{\Lambda}
\vert^{-s} \\ &=\left(\prod_{i=1}^g(1-t^{2f_i})\right)
\zeta^\vartriangleleft_{\Zp^{2n}}(s) \left(
     {\sum_{\ell\in\Adm_{\bfe,\bff}}t^{2\sum_{i=1}^n\ell_i}\sum_{\mu\leq\lambda(\ell)}\alpha(\lambda(\ell),\mu;p)\left(p^{2n}t\right)^{\sum_{i=1}^n\mu_i}}
     \right).
\end{align*}
The last bracketed factor above is exactly $D^{\bfe, \bff}(p,t)$, and
our claim follows.
\end{proof}

Given $(\bfe,\bff)\in\N^g\times\N^g$ as above and a Dyck word
$w\in\mathcal{D}_{2n}$, we set
\begin{equation} \label{equ:basic.eq.split} D_w^{\bfe,\bff}(p,t) = \sum_{\mu
    \leq \lambda \atop w(\mu, \lambda) = w} \alpha(\lambda, \mu;p)
  \mbox{ }(p^{2n}t)^{\sum_{i=1}^n\mu_i} \left( \sum_{\ell \in
    \Adm_{\bfe, \bff} \atop \lambda(\ell) = \lambda}
  t^{2\sum_{i=1}^n\ell_i} \right),
\end{equation}
so that $D^{\bfe,\bff} = \sum_{w\in\mathcal{D}_{2n}}D^{\bfe,\bff}_w$ and
therefore
\begin{equation} \label{equ:dyck.split}
 \zideal_{L_p}(s) = \left( \prod_{i=1}^g (1 - t^{2 f_i}) \right) \zeta^\vartriangleleft_{\Z_p^{2n}}(s) \sum_{w \in \mathcal{D}_{2n}} D_w^{\bfe,\bff} (p,t).
\end{equation}
If $\bfe=\bf1$, then we write $D^{\bff}$ instead of $D^{\bfe,\bff}$
and $D_w^{\bff}$ instead of $D_w^{\bfe,\bff}$.

In the next section we compute explicit formulae for the generating
functions~$D_w^\bff$.  We work with the variables $p$ and $t$, but it
will be clear that the coefficients of the rational functions obtained
depend only on $\bff$ and~$w$.

\section{Computation of the functions
  $W^\vartriangleleft_{\bf1,\bff}(X,Y)$}

\subsection{A special case: completely split primes
  ($\bff=(1,\dots,1)$)}\label{subsec:totally.split} We start with the
computation of the functions $W^\vartriangleleft_{\bfo,\bfo}(X,Y)$,
treating rational primes which split completely in $K$. Although this
case is subsumed in the general unramified case presented in
Subsection~\ref{sec:gen.unram}, we present it separately as it
illustrates our method and serves as a template for the general case.

Recall that by~\eqref{equ:dyck.split} it suffices to compute the
functions $D^{\bfo}_w$, indexed by Dyck words $w\in\mcD_{2n}$, that were defined
in~\eqref{equ:basic.eq.split}.  Recall that
$\Adm_{\bfo,\bfo}=\N_0^n$.
\begin{thm} \label{thm:tot.split} 
Let $w=\prod_{i=1}^r\left(\bfz^{L_i-L_{i-1}}\bfo^{M_i-M_{i-1}}\right)
\in \mathcal{D}_{2n}$ be a Dyck word and set $\mathcal{L} = \{ L_1,
\dots, L_{r-1} \}\subseteq [n-1]$.  Then
\begin{multline*}
  D_w^{\mathbf{1}} (p,t) = \binom{n}{\mathcal{L}} \prod_{i = 1}^r
  \binom{L_i - M_{i-1}}{L_i - M_i}_{p^{-1}} \prod_{i = 1}^r I_{L_i -
    L_{i-1}}(1; y_{L_{i-1}+1}, \dots, y_{L_i}) \, \cdot
  \\ \left(\prod_{i = 1}^{r-1} I_{M_i - M_{i-1}}^{\circ} (p^{-1};
  x_{M_{i-1} + 1}, \dots, x_{M_i}) \right) I_{n - M_{r-1}}(p^{-1};
  x_{M_{r-1} + 1}, \dots, x_n),
\end{multline*}
with the numerical data
\begin{alignat}{2}
x_j & = p^{j(2n + L_i - j)} t^{2L_i + j} &&\quad\text{ for } \quad j
\in \,]M_{i-1}, M_i] \label{equ:num.data.tot.split.x},\\ y_j & =
    p^{(2n -M_{i-1} + j)M_{i-1}} t^{2j + M_{i-1}} &&\quad\text{ for }
    \quad j \in \,]L_{i-1}, L_i].\label{equ:num.data.tot.split.y}
\end{alignat}
\end{thm}
\begin{proof}
Our starting point is the defining
expression~\eqref{equ:basic.eq.split} for the functions $D_w^{\bfo}$.
Note that summing over all partitions $\mu \leq \lambda$ such that
$w(\mu, \lambda) = w$ is equivalent to summing over all the successive
differences $r_j$ and $s_j$, for $j\in[n]$, as defined in
\eqref{equ:defrj} and~\eqref{equ:defsj}.  Observe that
\begin{align} \label{equ:musum}
\mu_1 + \cdots + \mu_n & =  \sum_{j = 1}^n j r_j + \sum_{i = 1}^{r - 1} M_i (s_{L_i + 1} + \cdots + s_{L_{i+1}}), \\ \nonumber
\lambda_1 + \cdots + \lambda_n & =  \sum_{j = 1}^n j s_j + \sum_{i = 1}^r L_i (r_{M_{i-1} + 1} + \cdots + r_{M_i}).
\end{align}

Given a vector $\mathbf{v}=(v_1, \dots, v_n)\in\N_0^n$ we set, for
each $i \in [r]$,
\begin{align} \label{equ:def.jir}
\supp^M_i({\mathbf{v}}) &= \{ j \in [M_i - M_{i-1} - 1] \mid v_{M_{i-1} + j}
> 0 \},\\ \nonumber \supp^L_i({\mathbf{v}}) &= \{ j \in [L_i - L_{i-1} -1]
\mid v_{L_{i-1} + j} > 0 \}.
\end{align}  
In practice, $\mathbf{v}$ will be one of the vectors of successive
differences $\bfr = (r_1, \dots, r_n)$ or~$\bfs = (s_1, \dots, s_n)$.
Given a pair of partitions $\mu \leq \lambda$, recall the sets
$J_i^\mu$ and $J_i^\lambda$ that were defined in \eqref{equ:jimu} for
each $i \in [r]$.  It is easy to see that, for every $i \in [r]$, we
have
\begin{equation*}
\supp^M_i(\mathbf{r}) = M_i - M_{i-1} - J_i^\mu \quad \textrm{ and }
\quad \supp^L_i(\mathbf{s}) = L_i - L_{i-1} - J_i^\lambda,
\end{equation*}
in the notation of Remark \ref{rem:coxeter}.  It follows from the same
remark that
\begin{equation} \label{equ:rev.binoms}
\binom{M_i - M_{i-1}}{J_i^\mu}_{p^{-1}} = \binom{M_i -
  M_{i-1}}{\supp^M_i(\mathbf{r})}_{p^{-1}} \quad \textrm{ and } \quad \binom{L_i -
  L_{i-1}}{J_i^\lambda} = \binom{L_i -
  L_{i-1}}{\supp^L_i(\mathbf{s})}.
\end{equation}
We let $\delta_{ij}$ be the usual Kronecker delta function.
Substituting the results of Lemmata \ref{lem:lambda},
\ref{lem:mu.split}, and \ref{lem:lambda.p.split} into the right-hand
side of \eqref{equ:basic.eq.split}, rewriting the expressions in terms
of the $r_j$ and $s_j$, and using \eqref{equ:rev.binoms}, we find that
the summands are products of $2r$ factors. For each $i\in[r]$, there
are two factors, each involving either the terms $r_j$, where $M_{i-1}
+ 1 \leq j \leq M_i$, or the terms $s_j$, where $L_{i-1} + 1 \leq j
\leq L_i$.  More precisely, the formula \eqref{equ:basic.eq.split} for
$D_w^{\mathbf{1}}(p,t)$ splits into a product as follows:

\begin{equation} \label{equ:dwpt.product}
D_w^{\mathbf{1}}(p,t) = \binom{n}{\mathcal{L}} \prod_{i = 1}^r \binom{L_i - M_{i - 1}}{L_i - M_i}_{p^{-1}} \cdot \prod_{i = 1}^{r} A_i B_i,
\end{equation}
where, for $i\in[r]$,
\begin{align*}
A_i & = \sum_{r_{M_{i-1} + 1} = 0}^{\infty} \cdots \sum_{r_{M_{i} - 1}
  = 0}^{\infty} \sum_{r_{M_i} = 1-\delta_{ir}}^{\infty} \binom{M_i -
  M_{i-1}}{\supp^M_i({\mathbf{r}})}_{p^{-1}} \prod_{j = M_{i-1} + 1}^{M_i}
\left( p^{(j(L_i - j) + 2nj)} t^{(2L_i + j)} \right)^{r_j} \\ 
B_i & =
\sum_{s_{L_{i-1} + 1} = 0}^{\infty} \cdots \sum_{s_{L_i} = 0}^{\infty}
\binom{L_i - L_{i-1}}{\supp^L_i({\mathbf{s}})} \prod_{j = L_{i-1} + 1}^{L_i}
\left( p^{(2n -M_{i-1} +j) M_{i-1}} t^{(2j + M_{i-1} )} \right)^{s_j}.
\end{align*}

We now show that all of the factors $A_i$ and $B_i$ are products of
Igusa functions and Gaussian binomial coefficients.  Given $i\in [r]$
and $I \subseteq [L_i - L_{i-1} - 1]$, we define $\mathbf{S}^i(I)$ to
be the set of vectors $\mathbf{s}^i = (s_{L_{i-1}+1}, \dots, s_{L_i})
\in \N_0^{L_i - L_{i-1}}$ such that $s_j = 0$ unless $j \in \{ L_{i-1}
+ k \mid k \in I \} \cup \{ L_i \}$. With the numerical data defined
in \eqref{equ:num.data.tot.split.y}, we have
\begin{align*}
\lefteqn{B_i }\quad&= \sum_{I \subseteq [L_i - L_{i-1} - 1]}
\binom{L_i - L_{i-1}}{I} \, \sum_{\mathbf{s}^i \in \mathbf{S}^i(I)}
\quad \prod_{j\in (I + L_{i-1}) \cup\{L_i\}} \left(p^{(2n -M_{i-1} +j)
  M_{i-1}} t^{2j + M_{i-1}}\right)^{s_j} \\ &= \sum_{I \subseteq
  [L_i - L_{i-1} - 1]} \binom{L_i - L_{i-1}}{I} \left( \prod_{\iota
  \in I} \left( \sum_{s_{L_{i-1} + \iota} = 1}^{\infty} (y_{L_{i-1} +
  \iota})^{s_{L_{i-1} + \iota}} \right) \right)\sum_{s_{L_i} =
  0}^{\infty} (y_{L_i})^{s_{L_i}} \\ &= \frac{1}{1 - y_{L_i}} \sum_{I
  \subseteq [L_i - L_{i-1} - 1]} \binom{L_i - L_{i-1}}{I} \prod_{\iota
  \in I} \frac{y_{L_{i-1} + \iota}}{1 - y_{L_{i-1} + \iota}} \\ &=
I_{L_i - L_{i-1}}(1;y_{L_{i-1} + 1}, \dots, y_{L_i}),
\end{align*}
where the $y_{j}$ are as defined in the statement of the theorem.  

Analogously one shows that, with the numerical data defined
in~\eqref{equ:num.data.tot.split.x},
\begin{align*}
A_i = \begin{cases} I_{M_i - M_{i-1}}^{\circ} (p^{-1}; x_{M_{i-1} +
    1}, \dots, x_{M_i}) &\text{ for }i\in[r-1],\\ I_{n -
    M_{r-1}}(p^{-1}; x_{M_{r-1} + 1}, \dots, x_n)& \text{ for }i=r. 
\end{cases} 
\end{align*}
This completes the proof.  
\end{proof}

\begin{exm} 
Suppose that $n=g=3$ and $\bfe = \bff = (1,1,1)$.  In other words, $K$
is a cubic number field in which the prime $p$ is totally split.  The
corresponding zeta factor was obtained in Taylor's PhD~thesis by an
involved computation with cone integrals \cite[Theorem 15]{Taylor/01};
the formula is reproduced in \cite[Theorem 2.5]{duSWoodward/08}.  We
show how to recover it from Theorem~\ref{thm:tot.split}.

Recall that $\mcD_{6} =
\{\bfz\bfz\bfz\bfo\bfo\bfo,\bfz\bfz\bfo\bfz\bfo\bfo,\bfz\bfz\bfo\bfo\bfz\bfo,
\bfz\bfo\bfz\bfz\bfo\bfo,\bfz\bfo\bfz\bfo\bfz\bfo\}$.  We denote these
Dyck words by $A$, $B$, $C$, $D$, and $E$, respectively.  Writing out
the Igusa functions and noting that $I_\wo(1; t^2, t^4, \dots, t^{2
  \wo}) = \frac{1}{(1 - t^2)^\wo}$ for all $\wo \in \N$ (see
Lemma~\ref{lem:il}), we obtain the following formulae for
$D_w^\mathbf{1}(p,t)$. Here we use the notation $\gp{x} =
\frac{x}{1-x}$ and $\gpzero{x} = \frac{1}{1-x}$.
\begin{center}
{\renewcommand{\arraystretch}{1.4}
\begin{tabular}{| l | l |} \hline
$w$ & $D_w^\mathbf{1}(p,t)$ \\ \hline \hline A & $\gpzero{p^{18}t^9}
  \left( 1 + \binom{3}{1}_{p^{-1}}\left(\gp{p^{14}t^8} +
  \gp{p^8t^7}\right) +
  \binom{3}{1,2}_{p^{-1}}\gp{p^{14}t^8}\gp{p^8t^7}\right)
  \frac{1}{(1-t^2)^3}$ \\ B & $3 \gpzero{p^{18} t^9} \left( 1 +
  \binom{2}{1}_{p^{-1}}\gp{p^{14}t^8}\right) \gpzero{p^8t^7}
  \binom{2}{1}_{p^{-1}} \gp{p^7t^5}\frac{1}{(1-t^2)^2}$ \\ C & $3
  \gpzero{p^{18}t^{9}}\gpzero{p^{14}t^8}\gp{p^{12}t^6}\left(1+\binom{2}{1}_{p^{-1}}\gp{p^7t^5}\right)\frac{1}{(1-t^2)^2}$
  \\ D & $3 \gpzero{p^{18}t^{9}}\left(1 +
  \binom{2}{1}_{p^{-1}}\gp{p^{14}t^8}\right)
  \gpzero{p^8t^7}\left(1+2\gp{p^7t^5}\right) \gp{p^6t^3}
  \frac{1}{1-t^2}$ \\ E & $6
  \gpzero{p^{18}t^{9}}\gpzero{p^{14}t^8}\gp{p^{12}t^6}\gpzero{p^7t^5}\gp{p^6t^3}\frac{1}{1-t^2}$
  \\ \hline
\end{tabular}
}
\end{center}

Adding these five functions and multiplying the sum by $ (1 -
t^2)^3\zeta^\vartriangleleft_{\Z_p^6}(s)$, as prescribed by \eqref{equ:dyck.split}, we
indeed obtain Taylor's formula.
\end{exm}

As a further application of Theorem \ref{thm:tot.split}, we recover,
in Example~\ref{exm:luke}, the function dealing with primes that are
totally split in a quartic number field;
Woodward~\cite[Theorem~2.6]{duSWoodward/08} computed it by different
means.  For $n \geq 5$ the formulae we obtain are new.

\subsection{The general unramified case}\label{sec:gen.unram}
From now on, we fix $g\in\N$ and a vector $\bff =
(f_1,\dots,f_g)\in\N_0^g$ such that~$\sum_{i=1}^gf_i = n$. We aim to
compute the functions $W^\vartriangleleft_{\bfo,\bff}(X,Y)$. The
computation in this case is similar to the one carried out in the
totally split case ($\bff = \bf1$) in
Subsection~\ref{subsec:totally.split}, which it generalizes.  Recall
from \eqref{equ:dyck.split} and \eqref{equ:basic.eq.split} that
\begin{equation} \label{equ:dyck.unram}
 \zideal_{L_p}(s) = \left( \prod_{i=1}^g (1 - t^{2 f_i}) \right)
 \zeta^\vartriangleleft_{\Z_p^{2n}}(s) \sum_{w \in \mathcal{D}_{2n}} D_w^\bff (p,t),
\end{equation}
where, for each Dyck word $w\in\mathcal{D}_{2n}$,
\begin{equation} \label{equ:basic.eq.unram}
 D_w^\bff (p,t) = \sum_{\mu \leq \lambda \atop w(\mu, \lambda) = w}
 \alpha(\lambda, \mu;p) \mbox{ }(p^{2n} t)^{\sum_{i=1}^n \mu_i} \left(
 \sum_{\ell \in \Adm_{\mathbf{1}, \bff} \atop \lambda(\ell) = \lambda}
 t^{2\sum_{i=1}^n \ell_i} \right).
\end{equation}

In the special case $\bff=\bfo$ we have $\Adm_{\mathbf{1}, \mathbf{1}}
= \N_0^n$.  Then the sum inside the parentheses on the right-hand side
of \eqref{equ:basic.eq.unram} is $\beta(\lambda)
t^{2\sum_{i=1}\lambda_i}$, and this quantity is easily computed,
e.g.\ by means of Lemma \ref{lem:lambda}.  Thus in the computations in
Subsection \ref{subsec:totally.split} we could view the right-hand
side of \eqref{equ:basic.eq.unram} as a sum over pairs of partitions
$(\mu, \lambda)$ satisfying certain conditions.

The additional complication introduced when considering general $\bff$ is
that we must take into account the structure of $\Adm_{\mathbf{1},
  \bff}$.  The solution to the combinatorial problem of computing how
many {\emph{admissible}} $n$-tuples $\ell$ give rise to a given
partition $\lambda$ is not nearly as clean as Lemma \ref{lem:lambda}.
We avoid this issue by summing directly over pairs $(\ell, \mu)$, where $\ell
\in \Adm_{\mathbf{1}, \bff}$ and $\mu$ is a partition such that $\mu \leq
\lambda(\ell)$. 

\subsection{A refinement of the sums $D_w^\bff$}
We require precise control over the relation between admissible
$n$-tuples $\ell \in \Adm_{\mathbf{1}, \bff}$ and the corresponding
partitions~$\lambda(\ell)$.  For every $i \in [g]$, we have $C_i =
\sum_{j = 1}^i f_j$, as defined at the beginning of
Subsection~\ref{sec:lattices}.  Observe that there is a natural
bijection
\begin{equation}
\psi :\Adm_{\mathbf{1}, \bff} \rightarrow \N_0^g, \quad \nonumber \ell
\mapsto (\ell_{C_1}, \ell_{C_2}, \dots, \ell_{C_g}).
\end{equation}
The $g$-tuple $\psi(\ell)$ naturally gives rise to a weak ordering
$v_\ell = (\sigma_{\ell}, J_{\ell}) \in \WO_g \subseteq S_g \times
2^{[g-1]}$, obtained by arranging the components of $\psi(\ell)$ in
non-ascending order.  For instance, $\ell_{C_{\sigma_{\ell}(1)}}$ is
maximal among the components of $\psi(\ell)$ and
$\ell_{C_{\sigma_{\ell}(g)}}$ is minimal.  It is easy to express the
partition $\lambda(\ell)$ in terms of $v_{\ell}$.  Indeed, if we set
$C^\ell_i = \sum_{j = 1}^i f_{\sigma_{\ell}(j)}$ for $i \in [g]$, then
\begin{equation} \label{equ:lambdaell} \lambda(\ell)_j =
  \ell_{C_{\sigma_{\ell}(i)}} {\mbox{ }} \mathrm{if} \mbox{ } j \in\,
  ]C_{i-1}^\ell, C_i^\ell].
\end{equation}

Now fix a Dyck word $w \in \mathcal{D}_{2n}$; we compute $D_w^\bff$ by
partitioning the right-hand side of \eqref{equ:basic.eq.unram} into
summands parameterized by $\WO_g$.  Indeed, given $v \in \WO_g$, we
define
\begin{equation} \label{equ:dwv} D_{w,v}^\bff (p,t) = \sum_{\ell \in \Adm_{\mathbf{1}, \bff}
    \atop v_\ell = v} \sum_{\mu \leq \lambda(\ell) \atop w(\mu,
    \lambda(\ell)) = w} \alpha(\lambda(\ell), \mu;p) \mbox{ } (p^{2n}t)^{\sum_{i=1}^n \mu_i} t^{2\sum_{i=1}^n \ell_i},
\end{equation}
so that 
$$ D_w^\bff (p,t) = \sum_{v\in \WO_g}D_{w,v}^\bff (p,t).$$

The functions $D_{w,v}^\bff$ are computed in Lemma~\ref{lem:dwvpt}.
Afterwards we will see that they can be grouped together into sums
that are expressible in terms of the generalized Igusa functions
defined in Definition~\ref{dfn:gen.igusa}; cf.~\eqref{def:DwA} and
Theorem~\ref{thm:main.thm.unram}.

\begin{rem} \label{rmk:dwvpt} 
 We say a few words about the motivation behind the definition of the
 functions $D_{w,v}^\bff$.  The condition $\ell \in \Adm_{\mathbf{1},
   \bff}$ amounts to the fact that the partition $\lambda(\ell)$ is
 made up of $g$ ``blocks,'' each consisting of $f_1, f_2, \dots, f_g$
 equal parts.  The weak ordering $v_\ell = (\sigma_v,J_v) \in \WO_g$
 keeps track of the situation where the largest parts of
 $\lambda(\ell)$ are the $f_{\sigma_v(1)}$ equal parts coming from the
 prime $\p_{\sigma_v(1)}$, that the next-largest parts (possibly of
 equal sizes to the parts coming from $\p_{\sigma_v(1)}$) come from
 $\p_{\sigma_v(2)}$, etc.  Moreover, $J_v$ specifies when the parts
 coming from two different prime ideals are equal.  Thus, $v_\ell$
 tells us exactly which differences between adjacent blocks of parts
 of $\lambda(\ell)$ are zero and which are positive; this information
 is essential to our method.
\end{rem}

Our first task is to see when the set of pairs $(\mu, \ell)$ over
which the sum \eqref{equ:dwv} runs is non-empty.  Let
$w=\prod_{i=1}^r\left(\bfz^{L_i-L_{i-1}}\bfo^{M_i-M_{i-1}}\right)$.
The condition $w(\mu, \lambda(\ell)) = w$ ensures in particular that
$\lambda(\ell)_{L_i} > \lambda(\ell)_{L_i + 1}$ for all $1 \leq i \leq
r - 1$.  By \eqref{equ:lambdaell} this in turn implies that for each
$i \in [r-1]$ we have $L_i = C^\ell_{t_i}$ for some $t_i \in [g]$, and
moreover that $t_i \in J_\ell$.  Observe that this is a condition on
$v_\ell$; if it is satisfied, then we say that $v$ is \emph{compatible
  with}~$w$.  It is easy to see that $v$ is compatible with $w$ if and
only if $D_{w,v}^\bff(p,t)$ is a non-vacuous sum. It is useful to
rephrase the condition above as follows.

\begin{dfn}\label{def:comp}
  By a {\emph{set partition of}} $[g]$ we mean an ordered collection
  $\mathcal{A} = (\mathcal{A}_1, \dots, \mathcal{A}_s)$ of pairwise
  disjoint non-empty subsets $\mathcal{A}_1, \dots \mathcal{A}_s
  \subseteq [g]$ such that $\bigcup_{i = 1}^s \mathcal{A}_i = [g]$.
  Let
  $w=\prod_{i=1}^r\left(\bfz^{L_i-L_{i-1}}\bfo^{M_i-M_{i-1}}\right)\in\mathcal{D}_{2n}$.
  We say that $\mathcal{A}$ is \emph{compatible} with $w$ if $s = r$, and for each $i
  \in[r]$ we have $\sum_{j \in \mathcal{A}_i} f_j = L_i - L_{i-1}$.
  We denote by $\mathcal{P}_w$ the set of set partitions of $[g]$ that
  are compatible with $w$.
\end{dfn}

It is clear that a weak ordering $v=(\sigma_v,J_v) \in \WO_g$ is
compatible with a Dyck word $w\in\mathcal{D}_{2n}$ if and only if
there exists a sequence $0 = t_0 < t_1 < t_2 < \dots < t_{r-1} < t_r =
g$ such that $\{ t_1, \dots, t_{r-1} \} \subseteq J_v$ and such that
the set partition $\mathcal{A} = (\mathcal{A}_1, \dots, \mathcal{A}_r
)$ of $[g]$ is compatible with $w$, where for each $k \in [r]$,
\begin{equation*}
\mathcal{A}_k = \{ \sigma_v(t_{k-1} + 1), \dots, \sigma_v (t_k) \}.
\end{equation*}
If such a sequence $\{ t_k \}$ exists, then it is unique, and we may
denote $\mathcal{A} = \mathcal{A}(w,v)$.

Now, given a set partition $\mathcal{A} = (\mathcal{A}_1, \dots,
\mathcal{A}_r)$ compatible with a Dyck word $w$, we want to
parameterize all the weak orderings $v$ such that $\mathcal{A}(w,v) =
\mathcal{A}$.  For all $i \in [r]$, define $t_i = \sum_{k = 1}^i |
\mathcal{A}_k |$.  Let the elements of $\mathcal{A}_i$ be $a_1^{(i)} <
\cdots < a_{t_i - t_{i-1}}^{(i)}$.

Consider the map
\begin{align} \label{equ:phialpha}{\vphi}: \prod_{i = 1}^r
  \mathrm{WO}_{t_i - t_{i-1}} & \to \mathrm{WO}_g \\ \nonumber
  \mathbf{v} = ((\sigma_i, J_i))_i & \mapsto
  (\sigma_{\vphi(\mathbf{v})}, J_{\vphi(\mathbf{v})}),
\end{align}
where $\sigma_{\vphi(\mathbf{v})} \in S_g$ is given by
$\sigma_{\vphi(\mathbf{v})}(t_{i-1} + j) = a^{(i)}_{\sigma_i(j)}$ for
all $i \in [r]$ and $j \in [t_i - t_{i-1}]$, and
$J_{\vphi(\mathbf{v})}$ is the disjoint union
$$ J_{\vphi(\mathbf{v})} = \{ t_1, \dots, t_{r-1} \} \cup \bigcup_{i =
  1}^r \{ t_{i-1} + j \mid j \in J_i \}.$$
It is easy to see that $\vphi$ is injective and that its image
consists precisely of the weak orderings $v \in \mathrm{WO}_g$ such
that $\mathcal{A}(w,v) = \mathcal{A}$.

\begin{lemma} \label{lem:dwvpt}
 Let
 $w=\prod_{i=1}^r\left(\bfz^{L_i-L_{i-1}}\bfo^{M_i-M_{i-1}}\right)\in\mathcal{D}_{2n}$.
 Suppose that $v \in \WO_g$ is a weak ordering compatible with $w$.
 Let $\mathcal{A} = \mathcal{A}(w,v)$, let $t_i$ and $a_j^{(i)}$ be
 defined as above for all $i \in [r]$ and all $j \in [t_i - t_{i-1}]$,
 and let $\mathbf{v} = (v_1, \dots, v_r) \in \prod_{i = 1}^r
 \mathrm{WO}_{t_i - t_{i-1}}$ be such that $\vphi(\mathbf{v}) = v$.
 Consider the chains $\phi(v_i) \in P_{t_i - t_{i-1}}$ as
 in~\eqref{equ:poset.isom}.  Then
\begin{multline*}
  D_{w,v}^\bff (p,t) = \prod_{i = 1}^r \binom{L_i - M_{i-1}}{L_i -
    M_i}_{p^{-1}} \prod_{i = 1}^r \left( \frac{1}{1 - y^{(i)}_{[t_i - t_{i-1}]}}
    \prod_{\mcI \in \phi(v_i)} \frac{y^{(i)}_{\mcI}}{1 -
      y^{(i)}_{\mcI}} \right) \, \cdot \\ \prod_{i = 1}^{r-1}
  I_{M_i - M_{i-1}}^{\circ} (p^{-1}; x_{M_{i-1} + 1}, \dots, x_{M_i})
  \cdot I_{n - M_{r-1}}(p^{-1}; x_{M_{r-1} + 1}, \dots, x_n),
\end{multline*}
where for each $i \in [r]$ and for each subset $\mcI \subseteq [t_i -
t_{i-1}]$ we set $\varepsilon^{(i)}(\mcI) = L_{i-1} + \sum_{j \in
  \mcI} f_{a^{(i)}_j}$ and define the numerical data
\begin{alignat*}{2}
x_j & = p^{j(2n + L_i - j)} t^{2L_i + j} &&\quad\text{ for } j \in\, ]M_{i-1}, M_i], \\ y^{(i)}_{\mcI} & = p^{(2n -M_{i-1} +
  \varepsilon^{(i)}(\mcI))M_{i-1}} t^{2 \varepsilon^{(i)}(\mcI) +
  M_{i-1}} &&\quad\text{ for } \mcI \subseteq [t_i - t_{i-1}].
\end{alignat*}
\end{lemma}

\begin{proof}
  The relevant computations are very similar to those in the proof of
  Theorem~\ref{thm:tot.split}.  If $\ell \in \mathrm{Adm}_{\mathbf{1},
    \bff}$ and $\mu \leq \lambda(\ell)$ is a partition such that
  $w(\mu, \lambda(\ell)) = w$, then define the successive differences
  $\{ r_j, s_j \mid j \in [n] \}$ just as in \eqref{equ:defrj} and
  \eqref{equ:defsj}.  It follows from \eqref{equ:lambdaell} and from
  unraveling the definitions that the conditions $\ell \in
  \mathrm{Adm}_{\mathbf{1}, \bff}$ and $v_\ell = v$ impose the
  following conditions on the $s_j$:
\begin{enumerate}
\item[(1)] For all $i \in [r]$, we have $s_{L_i} = s_{\varepsilon^{(i)}([t_i - t_{i-1}])} \geq 0$.
\item[(2)] For all $i \in [r]$ and all $\mcI \in \phi(v_i)$, we have $s_{\varepsilon^{(i)}(\mcI)} > 0$.
\item[(3)] All other $s_j$ vanish.
\end{enumerate} 
  
Note that \eqref{equ:musum} expresses $\sum_{i = 1}^n \mu_i$ and
$\sum_{i = 1}^n \lambda_i$ in terms of the successive differences
$s_j$ and~$r_j$, whereas \eqref{equ:rev.binoms} and Lemmas
\ref{lem:mu.split} and \ref{lem:lambda.p.split} imply that
\begin{align*}
\alpha(\lambda(\ell),& \mu; p)  =  \left( \prod_{i = 1}^r \binom{L_i - M_{i-1}}{L_i - M_i}_{p^{-1}} \binom{M_i - M_{i-1}}{\supp^M_i(\mathbf{r})}_{p^{-1}} \right) \, \cdot \\ 
& p^{\sum_{i = 1}^r \left( \sum_{j = 1}^{M_i - M_{i-1}} (M_{i-1} + j)(L_i - M_{i-1} - j) r_{M_{i-1} + j} + \sum_{j = 1}^{L_i - L_{i-1}} M_{i-1} (L_{i-1} - M_{i-1} + j) s_{L_{i-1} + j} \right)},
\end{align*}
where the sets $\supp^M_i(\bfr) \subseteq [M_i - M_{i-1} - 1]$ are
defined in \eqref{equ:def.jir}.  Substituting all this into
\eqref{equ:dwv} and observing that some of the $s_j$ vanish as
above, we obtain the decomposition
\begin{equation*}
D_{w,v}^{\bff}(p,t) =  \prod_{i = 1}^r \binom{L_i - M_{i - 1}}{L_i - M_i}_{p^{-1}} \cdot \prod_{i = 1}^{r} A_i B_i,
\end{equation*}
where the functions $A_i$ are defined as in \eqref{equ:dwpt.product}
and
\begin{align*}
  B_i &= \sum_{\mathbf{s}^i \in \mathbf{S}_v^i} \sum_{s_{L_i} = 0}^{\infty} \left( \prod_{\mcI
      \in \phi(v_i) \cup [t_i - t_{i-1}]} (p^{(2n -M_{i-1} +
      \varepsilon^{(i)}(\mcI)) M_{i-1}} t^{2 \varepsilon^{(i)}(\mcI) +
      M_{i-1} })^{s_{\varepsilon^{(i)}(\mcI)}} \right) \\ &=
  \frac{1}{1 - y^{(i)}_{[t_i - t_{i-1}]}} \prod_{\mcI \in \phi(v_i)}
  \frac{y^{(i)}_{\mcI}}{1 - y^{(i)}_{\mcI}},
\end{align*}
where the $y^{(i)}_{\mcI}$ are defined as in the statement of the
lemma.  Here, for each $i \in [r]$, we define $E_v^{i} = \{
\varepsilon^{(i)}(\mcI) \mid \mcI \in \phi(v_i) \}$ and let
$\mathbf{S}_v^i$ be the collection of vectors $\mathbf{s}^i = (s_k)_{k
  \in E_v^i} \in \Z^{E_v^i}$ such that $s_k \geq 1$ for all $k \in
E_v^i$.  The functions $A_i$ were already computed in the proof of
Theorem \ref{thm:tot.split}.
\end{proof}

The functions $D_{w,v}^\bff$ computed in Lemma~\ref{lem:dwvpt} split
$D_w^\bff$ into too many summands to be useful; in particular,
$D_{w,v}^\bff$ need not satisfy any functional equation.  Therefore we
introduce a coarser decomposition of $D_w^\bff$ as follows.  Given a
set partition $\mathcal{A} \in \mathcal{P}_w$ of $[g]$ that is
compatible with the Dyck word $w$, we define
\begin{equation}\label{def:DwA}
  D_{w, \mathcal{A}}^\bff = \sum_{v \in \WO_g \atop \mathcal{A}(w,v) =
    \mathcal{A}} D_{w,v}^\bff.
\end{equation}

We will prove in Section \ref{sec:functeq} that $D_{w,
  \mathcal{A}}^\bff$ satisfies a functional equation whose symmetry
factor is independent of $w$ and $\mathcal{A}$;
cf.\ Proposition~\ref{pro:funeq.DwA}.  Recall that
\eqref{equ:dyck.unram} implies that
\begin{equation*}
  \zideal_{L_p}(s) = \prod_{i = 1}^g (1 - t^{2f_i}) \cdot
  \zideal_{\Zp^{2n}}(s) \sum_{w \in \mathcal{D}_{2n}}
  \sum_{\mathcal{A} \in \mathcal{P}_w} D_{w, \mathcal{A}}^\bff(p,t).
\end{equation*}

\begin{thm} \label{thm:main.thm.unram} 
Let Let
$w=\prod_{i=1}^r\left(\bfz^{L_i-L_{i-1}}\bfo^{M_i-M_{i-1}}\right)\in\mathcal{D}_{2n}$
and $\mathcal{A} \in \mathcal{P}_w$.  As before, let $t_i = \sum_{k =
  1}^i | \mathcal{A}_k |$ for~$i\in[r]$.  Then,
\begin{multline*}
D_{w,\mathcal{A}}^\bff (p,t) = \prod_{i = 1}^r \binom{L_i -
  M_{i-1}}{L_i - M_i}_{p^{-1}} \prod_{i = 1}^r I_{t_i -
  t_{i-1}}^{\mathrm{wo}}(\mathbf{y}^{(i)}) \, \cdot \\ \prod_{i =
  1}^{r-1} I_{M_i - M_{i-1}}^{\circ} (p^{-1}; x_{M_{i-1} + 1}, \dots,
x_{M_i}) \cdot I_{n - M_{r-1}}(p^{-1}; x_{M_{r-1} + 1}, \dots, x_n),
\end{multline*}
where $\mathbf{y}^{(i)} = (y^{(i)}_\mcI)_{\mcI \in 2^{[t_i - t_{i-1}]}
  \setminus \{ \varnothing \} }$, and the numerical data are
\begin{alignat*}{2}
  x_j & =  p^{j(2n + L_i - j)} t^{2L_i + j} &&\quad\text{ for } \quad j \in\, ]M_{i-1}, M_i], \\
  y^{(i)}_\mcI & = p^{(2n -M_{i-1} + \varepsilon^{(i)}(\mcI))M_{i-1}}
  t^{2 \varepsilon^{(i)}(\mcI) + M_{i-1}}&&\quad\text{ for }\quad \mcI \in 2^{[t_i -
    t_{i-1}]} \setminus \{ \varnothing \}.
\end{alignat*}
Here $\varepsilon^{(i)}(\mcI)$ is defined as in the statement of Lemma \ref{lem:dwvpt}.
\end{thm}
\begin{proof}
  The weak orderings $v \in \WO_g$ such that $\mathcal{A}(w,v) =
  \mathcal{A}$ are parameterized by the $r$-tuples of weak orderings
  $(v_1, \dots, v_r) \in \WO_{t_1} \times \WO_{t_2 - t_1} \times
  \cdots \times \WO_{g - t_{r-1}}$ via the map $\vphi$ of
  \eqref{equ:phialpha}.  The claim is now immediate from Lemma
  \ref{lem:dwvpt} and Definition~\ref{dfn:gen.igusa} of the
  generalized Igusa functions $I_{t_i - t_{i-1}}^{\mathrm{wo}}
  (\mathbf{y}^{(i)})$.
\end{proof}

\begin{cor}\label{cor:inert}
 Suppose that $p$ is inert in $K$.  Then
$$ \zeta^\vartriangleleft_{L_p} (s) = \zeta^\vartriangleleft_{\Z_p^{2n}}(s) I_n (p^{-1}; x_1, \dots, x_n),$$
where $x_j = p^{j(3n-j)} t^{2n+j}$ for all $j \in [n]$.
\end{cor}
\begin{proof}
It follows from Lemma \ref{lem:gentype} that $\Adm_{(1), (n)}$
consists of all $\ell \in \N_0^n$ such that all the components of
$\ell$ are equal.  Thus $D_{w}$ vanishes unless $w$ is the ``trivial''
Dyck word $\mathbf{0}^n \mathbf{1}^n$.  Moreover, $g = 1$ and there is
only one set partition $\mathcal{A}$ of $[g]$.  Thus, Theorem
\ref{thm:main.thm.unram} reduces to the statement that
\begin{equation*}
\zideal_{L_p}(s) = (1 - t^{2n}) \zeta^\vartriangleleft_{\Z_p^{2n}}(s) \Iwo_{1}(y_{[1]})
I_n(p^{-1}; x_1, \dots, x_n),
\end{equation*}
where $x_j = p^{j(3n - j)} t^{2n + j}$ for $j\in[n]$ and $y_{[1]} =
t^{2n}$.  The result follows since $\Iwo_{1}(y_{[1]}) = \frac{1}{1 -
  t^{2n}}$.
\end{proof}

\begin{rem}
Corollary~\ref{cor:inert} is also easily obtained with the methods of
\cite{Voll/05}. For details see~\cite{SV2/14}.
\end{rem}

\begin{exm}
  Observe that if $p$ is totally split in $K$, then $f_1 = \cdots =
  f_n = 1$ and it is easy to see that $D_{w, \mathcal{A}}^\bff$ is
  independent of the set partition $\mathcal{A}$.  Since there are
  $\binom{n}{\mathcal{L}}$ partitions compatible with the Dyck word
  $w$ and since in this case $\varepsilon^{(i)}(\mcI) = L_{i-1} + |
  \mcI |$ for all $\mcI \subseteq [t_i - t_{i-1}]$, we recover Theorem
  \ref{thm:tot.split} in view of the relation between the generalized
  and ``standard'' Igusa functions given in Lemma \ref{lem:ig.genig}.
\end{exm}


\section{The functional equation} \label{sec:functeq}
We say that a rational
function $W(X,Y) \in \Q(X,Y)$ satisfies a functional equation with
symmetry factor $(-1)^a X^b Y^c$ if the following holds:
$$ W(X^{-1}, Y^{-1}) = (-1)^a X^b Y^c W(X,Y).$$ We refer to the triple
$(a,b,c)\in\N_0^3$ as the symmetry data of the functional equation.

In this section we prove that, if $p$ is unramified in $K$, then the
Euler factor $\zideal_{H(\mcO_K),p}(s)$ satisfies a functional
equation with symmetry data independent of~$p$.  Recall
Definition~\ref{dfn:gen.igusa} of the generalized Igusa zeta functions
$\Iwo_\wo(\bfX)$, for $\wo\in\N$ and variables $\bfX = (X_\mcI)_{\mcI
  \in 2^{[\wo]}\setminus\{\varnothing\}}$.
\begin{pro}\label{pro:funeq.gen.igusa}
For all $\wo\in\N$,
$$\Iwo_\wo(\bfX^{-1}) = (-1)^{\wo}X_{[\wo]}\Iwo_\wo(\bfX).$$
\end{pro}

\begin{proof}
Recall from Subsection~\ref{subsec:weak.order} the interpretation of
$\WO_\wo$ as the face complex $P_\wo$ of the boundary $D_\wo$ of the
$(\wo-1)$-simplex.  Let $\Delta(P_\wo)$ be the order complex
of~$P_\wo$. As a simplicial complex, $\Delta(P_\wo)$ is isomorphic to
the second barycentric subdivision of~$D_\wo$. The geometric
realization of $\Delta(P_\wo)$ is, of course, isomorphic to the
$(s-2)$-sphere $S^{s-2}$, as is the geometric realization
of~$P_\wo$. This implies that $P_\wo$ is Gorenstein$^*$;
cf.~\cite[Section~4]{Stanley/91}. Noting that $P_\wo$ has rank
$\wo-1$, \cite[Proposition~4.4]{Stanley/91} yields
$$\sum_{y\in P_\wo} \prod_{\mcI\in y}
\frac{X^{-1}_{\mcI}}{1-X^{-1}_{\mcI}} =(-1)^{\wo-1}\left( \sum_{y\in
  P_\wo} \prod_{\mcI\in y} \frac{X_{\mcI}}{1-X_{\mcI}}\right).$$ The
claim follows.

An alternative proof uses the interpretation of $\Iwo_\wo(\bfX)$ as
the fine Hilbert series of a face ring;
cf.\ Remark~\ref{rem:face.ring}. The proposition's statement follows
from \cite[Corollary~7.2]{Stanley/96}, noting that the reduced Euler
characteristic of the $(\wo-1)$-simplex vanishes.
\end{proof}

\begin{pro}\label{pro:funeq.igusa} For all $\wo\in\N$,
\begin{align*}
I_\wo(Y^{-1};\bfX^{-1})& = (-1)^\wo
X_{\wo}Y^{-\binom{\wo}{2}}I_\wo(Y;\bfX),\\ I^\circ_\wo(Y^{-1};\bfX^{-1})& =
(-1)^\wo X_\wo^{-1}Y^{-\binom{\wo}{2}}I^\circ_\wo(Y;\bfX).
\end{align*}
\end{pro}

\begin{proof} 
This follows from \cite[Theorem~4]{Voll/05}; note Remark
\ref{rem:igusa.fnct}.
\end{proof}

Let $w \in \mathcal{D}_{2n}$ be a Dyck word and let $\mathcal{A} \in
\mathcal{P}_w$ be a set partition of $[g]$ compatible with~$w$;
cf.~Definition \ref{def:comp}.  Recall the definition~\eqref{def:DwA}
of the function~$D_{w,\mathcal{A}}^\bff$.

\begin{pro} \label{pro:funeq.DwA}  
The function $D_{w,\mathcal{A}}^\bff$ satisfies the functional
equation
$$ D_{w,\mathcal{A}}^\bff(p^{-1}, t^{-1}) = (-1)^{g+n} p^{\frac{5n^2 -
      n}{2}} t^{5n} D_{w,\mathcal{A}}^\bff (p,t).$$
\end{pro}
\begin{proof}
  This is a straightforward computation using the formula for
  $D_{w,\mathcal{A}}^\bff$ from Theorem~\ref{thm:main.thm.unram}.
  Indeed, the Gaussian binomial coefficients clearly satisfy
\begin{equation*}
\binom{a}{b}_{Y^{-1}} = Y^{b(b-a)} \binom{a}{b}_Y.
\end{equation*}

Combining this with the functional equations provided by
Propositions~\ref{pro:funeq.gen.igusa} and~\ref{pro:funeq.igusa}, we
see that each of the factors on the right-hand side of the formula of
Theorem~\ref{thm:main.thm.unram} satisfies a functional
equation. Hence $D_{w,\mathcal{A}}^\bff$ also satisfies a functional
equation whose symmetry factor is
\begin{multline*}
\prod_{i = 1}^r p^{(L_i - M_i)(M_i - M_{i-1})} \cdot \prod_{i = 1}^r
(-1)^{| \mathcal{A}_i |} y^{(i)}_{[t_i - t_{i-1}]} \cdot \prod_{i = 1}^{r-1} (-1)^{M_i -
  M_{i-1}} p^{-\binom{M_i - M_{i-1}}{2}} x_{M_i}^{-1} \\ \cdot
(-1)^{n-M_{r-1}} p^{-\binom{n - M_{r-1}}{2}} x_n.
\end{multline*}
Noting that $\sum_{i = 1}^r | \mathcal{A}_i | = g$ and substituting
the values of $x_{M_i}$ and $y^{(i)}_{[t_i - t_{i-1}]}$ from
Theorem~\ref{thm:main.thm.unram}, a simple calculation yields the
claim.
\end{proof}

The following theorem is equivalent to Theorem~\ref{thm:main.intro}.

\begin{thm} \label{thm:funeq.unram}
Suppose that $p$ is unramified in $K$.  Then we have the functional
equation
$$ \zideal_{L_p}(s) |_{p \to p^{-1}} = (-1)^{3n} p^{\binom{3n}{2} -
  5ns} \zideal_{L_p} (s).$$
\end{thm}
\begin{proof}  
Consider the formula \eqref{equ:dyck.split} for~$\zideal_{L_p}(s)$.
The factor $\zeta^\vartriangleleft_{\Z_p^{2n}}(s) = \prod_{i = 0}^{2n-1} \frac{1}{1 -
  p^i t}$ satisfies a functional equation with symmetry factor
$(-1)^{2n}p^{\binom{2n}{2}} t^{2n}$, while $\prod_{i = 1}^g (1 -
t^{2f_i})$ satisfies a functional equation with symmetry factor
$(-1)^g t^{-2\sum_{i=1}^g f_i}$, which is equal to $(-1)^g t^{-2n}$ as
$p$ is unramified.  Combining these facts with Proposition
\ref{pro:funeq.DwA}, we see that $\zideal_{L_p}(s)$ satisfies a
functional equation with symmetry factor $ (-1)^{3n} p^{\binom{3n}{2}}
t^{5n}$, and this is our claim.
\end{proof}

\begin{rem}
 Conjecture \ref{conj:functeq} follows from the claim that the
 functions $D^{\bfe,\bff}(p,t)$ defined in \eqref{equ:basic.eq.split}
 all satisfy a functional equation and that the symmetry data are, up
 to sign, independent of the decomposition type~$(\bfe, \bff)$.
 Indeed, if
$$ D^{\bfe, \bff}(p^{-1}, t^{-1}) = (-1)^{g + n} p^{\frac{5n^2 -
    n}{2}} t^{5n} D^{\bfe, \bff}(p,t) $$ for all $(\bfe, \bff)$, then
Conjecture \ref{conj:functeq} follows from \eqref{equ:dyck.split} and
a computation analogous to that in the proof of Theorem
\ref{thm:funeq.unram}.
\end{rem}

\section{Examples}
In this section we present several instances of the
results of this paper.  Throughout the section we use the notation
$\gp{x} = \frac{x}{1-x}$ and $\gpzero{x} = \frac{1}{1-x}$.  Our
computations in the first example rely on the following fact.

\begin{lemma} \label{lem:il}
For all $\wo\in\N$,
$$ I_\wo (1; X, X^2, \dots, X^\wo) = \frac{1}{(1 - X)^\wo}.$$
\end{lemma}
\begin{proof}
Bringing the left-hand side to a common denominator, we observe that
\begin{equation} \label{equ:comden}
I_\wo (1; X,X^2,  \dots, X^\wo) = \frac{\sum_{I \subseteq [\wo - 1]} \binom{\wo}{I} \left( \prod_{i \in I} X^i \right) \left( \prod_{i \not\in I} (1 - X^i) \right)   } {\prod_{i = 1}^{\wo} (1 - X^i)}.
\end{equation}
By \eqref{equ:binom.descent} we have that 
\begin{equation*}
\binom{\wo}{I} = \sum_{\sigma \in S_\wo \atop \mathrm{Des}(\sigma) \subseteq I} 1.
\end{equation*}
Thus the numerator of the right-hand side of \eqref{equ:comden} may be
rearranged as follows:
\begin{align*}
\lefteqn{\sum_{\sigma \in S_\wo} \sum_{I \supseteq
    \mathrm{Des}(\sigma)} \left( \prod_{i \in I} X^i \right) \left(
  \prod_{i \not\in I} (1 - X^i) \right)}\\ &= \sum_{\sigma \in S_\wo}
  \left( \prod_{i \in \mathrm{Des}(\sigma)} X^i \right) \sum_{J
    \subseteq [\wo - 1] \setminus \mathrm{Des}(\sigma)} \prod_{j \in
    J} X^j \prod_{j \not\in J} (1 - X^j) \\ &= \sum_{\sigma \in
    S_\wo} \left( \prod_{i \in \mathrm{Des}(\sigma)} X^i \right) =
  \sum_{\sigma \in S_\wo} X^{\mathrm{maj}(\sigma)}.
\end{align*}
Here $\mathrm{maj}(\sigma) = \sum_{i \in \mathrm{Des}(\sigma)} i$ is
the major index, and the second equality follows because
\begin{equation*}
 \sum_{J \subseteq [\wo - 1] \setminus \mathrm{Des}(\sigma)} \prod_{j
   \in J} X^j \prod_{j \not\in J} (1 - X^j) = \prod_{j \in [h-1] \setminus \mathrm{Des}(\sigma)} (X^j + (1 - X^j))= 1.
\end{equation*}
However, we have
\begin{equation*}
  \sum_{\sigma \in S_\wo} X^{\mathrm{maj}(\sigma)} = \sum_{\sigma \in S_\wo} X^{\len (\sigma)} = \prod_{i = 1}^{\wo} \frac{1 - X^i}{1 - X}.
 \end{equation*}
Here the first equality is the equidistribution of Coxeter length and
major index \cite[(1.41)]{Stanley/12} and the second equality is
\cite[Corollary 1.3.13]{Stanley/12}.  By \eqref{equ:comden}, our claim
follows immediately.
\end{proof}

\begin{exm} \label{exm:luke}
Consider the case of $n = [K:\Q] = 4$ and $p$ totally split in $K$.
The set $\mathcal{D}_8$ is comprised of fourteen Dyck words, listed
here in lexicographical order.

\begin{center}
{\renewcommand{\arraystretch}{1}
\begin{longtable}{|c || r | l |} \hline
 & Dyck word & Overlap types of partitions $\mu\leq \lambda$ \\
\hline\hline
A & \bfz \bfz \bfz \bfz \bfo \bfo \bfo \bfo &$\lambda_1 \geq \lambda_2 \geq \lambda_3 \geq \lambda_4 \geq \mu_1 \geq \mu_2 \geq \mu_3 \geq \mu_4$ \\ \hline
B & \bfz \bfz \bfz \bfo \bfz \bfo \bfo \bfo & $\lambda_1 \geq \lambda_2 \geq \lambda_3 \geq \mu_1 > \lambda_4 \geq \mu_2 \geq  \mu_3 \geq \mu_4$ \\ \hline
C & \bfz \bfz \bfz \bfo \bfo \bfz \bfo \bfo &$\lambda_1 \geq \lambda_2 \geq \lambda_3 \geq \mu_1 \geq \mu_2 > \lambda_4 \geq \mu_3 \geq \mu_4$ \\ \hline
D & \bfz \bfz \bfz \bfo \bfo \bfo \bfz \bfo & $\lambda_1 \geq \lambda_2 \geq \lambda_3 \geq \mu_1 \geq \mu_2 \geq \mu_3 > \lambda_4 \geq \mu_4$ \\ \hline
E & \bfz \bfz \bfo \bfz \bfz \bfo \bfo \bfo & $\lambda_1 \geq \lambda_2 \geq \mu_1 > \lambda_3 \geq \lambda_4 \geq \mu_2 \geq \mu_3 \geq \mu_4$ \\ \hline
F & \bfz \bfz \bfo \bfz \bfo \bfz \bfo \bfo & $\lambda_1 \geq \lambda_2 \geq \mu_1 > \lambda_3 \geq \mu_2 > \lambda_4 \geq \mu_3 \geq \mu_4$ \\ \hline
G & \bfz \bfz \bfo \bfz \bfo \bfo \bfz \bfo & $\lambda_1 \geq \lambda_2 \geq \mu_1 > \lambda_3 \geq \mu_2 \geq \mu_3 > \lambda_4 \geq \mu_4$ \\ \hline
H & \bfz \bfz \bfo \bfo \bfz \bfz \bfo \bfo & $\lambda_1 \geq \lambda_2 \geq \mu_1 \geq \mu_2 > \lambda_3 \geq \lambda_4 \geq \mu_3 \geq \mu_4$ \\ \hline
I & \bfz \bfz \bfo \bfo \bfz \bfo \bfz \bfo & $\lambda_1 \geq \lambda_2 \geq \mu_1 \geq \mu_2 > \lambda_3 \geq \mu_3 > \lambda_4 \geq \mu_4$ \\ \hline
J & \bfz \bfo \bfz \bfz \bfz \bfo \bfo \bfo & $\lambda_1 \geq \mu_1 > \lambda_2 \geq \lambda_3 \geq \lambda_4 \geq \mu_2 \geq \mu_3 \geq \mu_4$ \\ \hline
K & \bfz \bfo \bfz \bfz \bfo \bfz \bfo \bfo & $\lambda_1 \geq \mu_1 > \lambda_2 \geq \lambda_3 \geq \mu_2 > \lambda_4 \geq \mu_3 \geq \mu_4$ \\ \hline
L & \bfz \bfo \bfz \bfz \bfo \bfo \bfz \bfo & $\lambda_1 \geq \mu_1 > \lambda_2 \geq \lambda_3 \geq \mu_2 \geq \mu_3 > \lambda_4 \geq \mu_4$ \\ \hline
M & \bfz \bfo \bfz \bfo \bfz \bfz \bfo \bfo & $\lambda_1 \geq \mu_1 > \lambda_2 \geq \mu_2 > \lambda_3 \geq \lambda_4 \geq \mu_3 \geq \mu_4$ \\ \hline
N & \bfz \bfo \bfz \bfo \bfz \bfo \bfz \bfo & $\lambda_1 \geq \mu_1 > \lambda_2 \geq \mu_2 > \lambda_3 \geq \mu_3 > \lambda_4 \geq \mu_4$ \\ \hline
\end{longtable}
}
\end{center}

Below we list the functions $D_w^{\mathbf{1}}(p,t)$, for
$w\in\mathcal{D}_8$, as obtained from Theorem~\ref{thm:tot.split}.  To
simplify the expressions, we use the fact that $I_\wo(1; t^2, \dots,
t^{2 \wo}) = \frac{1}{(1 - t^2)^\wo}$ by Lemma~\ref{lem:il}.  One
verifies easily that the sum of these fourteen functions, multiplied
by $ (1 - t^2)^4\zeta^\vartriangleleft_{\Z_p^8}(s)$ as in \eqref{equ:dyck.split},
agrees with the function computed in Woodward's thesis and stated
in~\cite[Theorem~2.6]{duSWoodward/08}.

{\allowdisplaybreaks
\begin{align*} 
 D^\bfo_A = &\frac{1}{(1-t^2)^4} I_4(p^{-1}; p^{11} t^9, p^{20}t^{10},
  p^{27}t^{11}, p^{32} t^{12}) \\
 D^\bfo_B = &\frac{4}{(1-t^2)^3} \binom{3}{2}_{p^{-1}} \gp{p^{10}t^7}
  \gpzero{p^{11} t^9} I_3(p^{-1}; p^{20} t^{10}, p^{27} t^{11}, p^{32}
  t^{12}) \\
D^\bfo_C = &\frac{4}{(1-t^2)^3} \binom{3}{1}_{p^{-1}}
  I_2^\circ(p^{-1}; p^{10} t^7, p^{18}t^{8}) \gpzero{p^{20} t^{10}}
  I_2(p^{-1}; p^{27} t^{11}, p^{32} t^{12}) \\
D^\bfo_D = &\frac{4}{(1-t^2)^3} I_3^\circ (p^{-1}; p^{10} t^7, p^{18}
  t^8, p^{24} t^9) \gpzero{p^{27} t^{11}} \gpzero{p^{32} t^{12}} \\
D^\bfo_E = &\frac{6}{(1-t^2)^2} \binom{2}{1}_{p^{-1}} \gp{p^9 t^5}
  I_2(1; p^{10} t^7, p^{11} t^9) I_3(p^{-1}; p^{20} t^{10}, p^{27}
  t^{11}, p^{32} t^{12}) \\
D^\bfo_F = &\frac{12}{(1-t^2)^2} \binom{2}{1}_{p^{-1}}^2 \gp{p^9 t^5}
  \gpzero{p^{10} t^7} \gp{p^{18} t^8} \gpzero{p^{20} t^{10}}
  I_2(p^{-1}; p^{27} t^{11}, p^{32} t^{12}) \\
D^\bfo_G = &\frac{12}{(1 - t^2)^2} \binom{2}{1}_{p^{-1}} \gp{p^9 t^5}
  \gpzero{p^{10} t^7} I_2^\circ(p^{-1}; p^{18} t^8, p^{24} t^9)
  \gpzero{p^{27} t^{11}} \gpzero{p^{32} t^{12}} \\
D^\bfo_H = &\frac{6}{(1 - t^2)^2} I_2^\circ(p^{-1}; p^9 t^5, p^{16}
  t^6) I_2(1; p^{18} t^8, p^{20} t^{10}) I_2(p^{-1}; p^{27} t^{11},
  p^{32} t^{12}) \\
D^\bfo_I = &\frac{12}{(1 - t^2)^2} I_2^\circ (p^{-1}; p^9 t^5, p^{16}
  t^6) \gpzero{p^{18} t^8} \gp{p^{24} t^9} \gpzero{p^{27} t^{11}}
  \gpzero{p^{32} t^{12}} \\
D^\bfo_J = &\frac{4}{1 - t^2} \gp{p^8 t^3} I_3(1; p^9 t^5, p^{10} t^7,
  p^{11} t^9) I_3(p^{-1}; p^{20} t^{10}, p^{27} t^{11}, p^{32}
  t^{12}) \\
D^\bfo_K = &\frac{12}{1 - t^2} \binom{2}{1}_{p^{-1}} \gp{p^8 t^3}
  I_2(1; p^9 t^5, p^{10} t^7) \gp{p^{18} t^8} \gpzero{p^{20} t^{10}}
  I_2(p^{-1}; p^{27} t^{11}, p^{32} t^{12}) \\
D^\bfo_L = &\frac{12}{1 - t^2} \gp{p^8 t^3} I_2(1; p^9 t^5, p^{10}
  t^7) I_2^\circ (p^{-1}; p^{18} t^8, p^{24} t^{9}) \gpzero{p^{27}
    t^{11}} \gpzero{p^{32} t^{12}} \\
D^\bfo_M = &\frac{12}{1 - t^2} \gp{p^8 t^3} \gpzero{p^9 t^5}
  \gp{p^{16} t^6} I_2(1; p^{18} t^8, p^{20} t^{10}) I_2(p^{-1}; p^{27}
  t^{11}, p^{32} t^{12}) \\
D^\bfo_N = &\frac{24}{1 - t^2} \gp{p^8 t^3} \gpzero{p^9 t^5}
  \gp{p^{16} t^6} \gpzero{p^{18} t^8} \gp{p^{24} t^9} \gpzero{p^{27}
    t^{11}} \gpzero{p^{32} t^{12}}.
\end{align*}
}
\end{exm}

\begin{exm}
Consider the case $n=[K : \Q] = 4$ and $p \mathcal{O}_K =
\mathfrak{p}_1 \mathfrak{p}_2$ with $\bff = (f_1, f_2) = (2,2)$.  In
this case,
$$ \mathrm{Adm}_{\mathbf{1}, \bff} = \left\{ \ell = (\ell_1, \ell_2, \ell_3, \ell_4) \in \N_0^4 \mid \ell_1 = \ell_2, \ell_3 = \ell_4 \right\}.$$  

The four parts of a partition $\lambda(\ell)$ arising from any $\ell
\in \mathrm{Adm}_{\mathbf{1}, \bff}$ necessarily split into two pairs,
with the parts in each pair being equal.  Only three of the fourteen
elements of $\mathcal{D}_8$ allow for this situation; these are the
Dyck words labeled A, E, and H in the chart given in
Example~\ref{exm:luke}.
  
Only one set partition of $[2]$ is compatible with the Dyck word $A$,
namely the set partition $\mathcal{A} = (\{1,2\})$.  An easy
computation shows $I_2^{\mathrm{wo}} (\mathbf{y}^{(1)}) =
\frac{1}{(1 - t^4)^2}$, and hence Theorem \ref{thm:main.thm.unram}
yields
$$D_{A}^{(2,2)} = D_{A, \mathcal{A}}^{(2,2)} = \frac{1}{(1 - t^4)^2} I_4(p^{-1}; p^{11} t^9, p^{20} t^{10}, p^{27} t^{11}, p^{32} t^{12}).$$

There are two set partitions of $[2]$ compatible with each of the Dyck
words E and H, namely $\mathcal{A}^\prime = (\{1\},\{2\})$ and
$\mathcal{A}^{\prime \prime} = (\{2\},\{1\})$.  Since the inertia
degrees of the two prime ideals lying over $p$ are
equal, $D_{w, \mathcal{A}}^{(2,2)}(p,t)$ is independent of the set
partition $\mathcal{A}$.  Now Theorem \ref{thm:main.thm.unram} gives
\begin{align*}
D^{(2,2)}_E = 2 D^{(2,2)}_{E, \mathcal{A}^\prime} = & \frac{2}{1 -
  t^4} \binom{2}{1}_{p^{-1}} \gp{p^9 t^5} \gpzero{p^{11} t^9}
I_3(p^{-1}; p^{20} t^{10}, p^{27} t^{11}, p^{32} t^{12})
\\ D^{(2,2)}_H = 2 D^{(2,2)}_{H, \mathcal{A}^\prime} = & \frac{2}{1 -
  t^4} I_2^\circ (p^{-1}; p^9 t^5, p^{16} t^6) \gpzero{p^{20} t^{10}}
I_2(p^{-1}; p^{27} t^{11}, p^{32} t^{12}).
\end{align*}
     
Adding these three functions and multiplying by $(1 -
t^4)^2\zeta^\vartriangleleft_{\Z_p^8}$ as in \eqref{equ:dyck.split}, we obtain
\begin{equation*}
\zeta^\vartriangleleft_{L_p}(s) = \zeta^\vartriangleleft_{\Z_p^8}(s) \zeta_p(11s - 27)
\zeta_p(10s - 20) \zeta_p(9s - 11) \zeta_p(5s - 9) \zeta_p(6s - 16)^2
\cdot P(p,t),
\end{equation*}
where
\begin{align*}
P(p,t)  =  &  p^{61} t^{35} + 2 p^{53} t^{30} - p^{53} t^{26} + p^{52} t^{30} - p^{52} t^{26} + p^{51} t^{26} - 
    p^{45} t^{25} + p^{44} t^{25} - \\ & p^{44} t^{21} + 2 p^{43} t^{25} - p^{43} t^{21} + p^{42} t^{25} - 
    p^{42} t^{21} - p^{37} t^{24} - p^{36} t^{24} + p^{36} t^{20} + \\ & p^{35} t^{24} - p^{35} t^{20} - 
    p^{35} t^{16} - p^{34} t^{16} + p^{33} t^{20} - p^{33} t^{16} - p^{28} t^{19} + p^{28} t^{15} - \\ &
    p^{27} t^{19} - p^{26} t^{19} - p^{26} t^{15} + p^{26} t^{11} + p^{25} t^{15} - p^{25} t^{11} - 
    p^{24} t^{11} - p^{19} t^{14} + \\ & p^{19} t^{10} - p^{18} t^{14} + 2 p^{18} t^{10} - p^{17} t^{14} + 
    p^{17} t^{10} - p^{16} t^{10} + p^{10} t^9 - p^9 t^9 + \\ & p^9 t^5 - p^8 t^9 + 2 p^8 t^5 +
    1.
\end{align*}
\end{exm}

\begin{exm}  
Let $[K : \Q] = 4$ and suppose $p \mathcal{O}_K = \p_1 \p_2$ with
$\bff = (f_1, f_2) = (3,1)$.  In this case, at least three of the four
parts of a partition $\lambda(\ell)$ arising from $\ell \in
\mathrm{Adm}_{\mathbf{1}, \bff}$ must be equal to each other, and only
the Dyck words A, B, C, D, and J allow for this.  In each of these
five cases, only one set partition $\mathcal{A}$ of $[2]$ is
compatible with the given Dyck word, namely $\mathcal{A} = \{ 1,2 \}$
for the word A, $\mathcal{A} = (\{1 \}, \{ 2 \})$ for the words B, C,
and D, and $\mathcal{A} = (\{ 2 \}, \{ 1 \})$ for the word J.  We
apply Theorem \ref{thm:main.thm.unram} to compute the zeta function.

For the word A, we observe that $(y^{(1)}_{\{ 1 \} }, y^{(1)}_{\{ 2 \} }, y^{(1)}_{\{ 1,2 \} }) = (t^6, t^2, t^8)$, and hence that 
$$I_2^{\mathrm{wo}}(\mathbf{y}^{(1)}) = \frac{1}{1 - t^8} \left( 1 + \frac{t^6}{1 - t^6} + \frac{t^2}{1 - t^2} \right) = \frac{1}{(1 - t^6)(1 - t^2)}.$$

Therefore, 
\begin{equation*}
D^{(3,1)}_A = \frac{1}{(1 - t^6)(1 - t^2)} I_4(p^{-1}; p^{11} t^9, p^{20} t^{10}, p^{27} t^{11}, p^{32} t^{12}).
\end{equation*}

Similarly, for the other relevant Dyck words we obtain:
\begin{align*}
D^{(3,1)}_B = & \binom{3}{2}_{p^{-1}} \gpzero{t^6} \gp{p^{10} t^7} \gpzero{p^{11} t^9} I_3(p^{-1}; p^{20} t^{10}, p^{27} t^{11}, p^{32} t^{12}) \\
D^{(3,1)}_C = & \binom{3}{1}_{p^{-1}} \gpzero{t^6} I_2^\circ (p^{-1}; p^{10} t^7, p^{18} t^8) \gpzero{p^{20} t^{10}} I_2(p^{-1}; p^{27} t^{11}, p^{32} t^{12}) \\
D^{(3,1)}_D = & \gpzero{t^6} I_3^\circ (p^{-1}; p^{10} t^7, p^{18} t^8, p^{24} t^9) \gpzero{p^{27} t^{11}} \gpzero{p^{32} t^{12}} \\
D^{(3,1)}_J = & \gpzero{t^2} \gp{p^8 t^3} \gpzero{p^{11} t^9} I_3(p^{-1}; p^{20} t^{10}, p^{27} t^{11}, p^{32} t^{12}).
\end{align*}

By \eqref{equ:dyck.split}, the sum of these five functions is
$\frac{\zeta^\vartriangleleft_{L_p} (s)}{(1 - t^6) (1 - t^2)
  \zeta^\vartriangleleft_{\Z_p^8}(s)}$.  The numerator of the zeta
function has 120 terms, so we do not reproduce it here.
\end{exm}

\def\cprime{$'$}
\providecommand{\bysame}{\leavevmode\hbox to3em{\hrulefill}\thinspace}
\providecommand{\MR}{\relax\ifhmode\unskip\space\fi MR }
\providecommand{\MRhref}[2]{%
  \href{http://www.ams.org/mathscinet-getitem?mr=#1}{#2}
}
\providecommand{\href}[2]{#2}


\begin{thebibliography}{10}

\bibitem{Bauer/13}
T.~Bauer, \emph{Computing normal zeta functions of certain groups},
  MSc~thesis, Bar-Ilan University, 2013.

\bibitem{BjoernerBrenti/05}
A.~Bj{\"o}rner and F.~Brenti, \emph{Combinatorics of {C}oxeter groups},
  Graduate Texts in Mathematics, vol. 231, Springer, New York, 2005.

\bibitem{Butler/87}
L.~M. Butler, \emph{A unimodality result in the enumeration of subgroups of a
  finite abelian group}, Proc. Amer. Math. Soc. \textbf{101} (1987), no.~4,
  771--775.

\bibitem{duSG/00}
M.~P.~F. du~Sautoy and F.~J. Grunewald, \emph{Analytic properties of zeta
  functions and subgroup growth}, Ann. of Math. (2) \textbf{152} (2000),
  793--833.

\bibitem{duSWoodward/08}
M.~P.~F. du~Sautoy and L.~Woodward, \emph{Zeta functions of groups and rings},
  Lecture Notes in Mathematics, vol. 1925, Springer-Verlag, Berlin, 2008.

\bibitem{Ezzat/14}
S.~Ezzat, \emph{Counting irreducible representations of the {H}eisenberg group
  over the integers of a quadratic number field}, J. Algebra \textbf{397}
  (2014), 609--624.

\bibitem{GSS/88}
F.~J. Grunewald, D.~Segal, and G.~C. Smith, \emph{Subgroups of finite index in
  nilpotent groups}, Invent. Math. \textbf{93} (1988), 185--223.

\bibitem{KlopschVoll/09}
B.~Klopsch and C.~Voll, \emph{Igusa-type functions associated to finite formed
  spaces and their functional equations}, Trans. Amer. Math. Soc. \textbf{361}
  (2009), no.~8, 4405--4436.

\bibitem{Neukirch/99}
J.~Neukirch, \emph{Algebraic number theory}, Grundlehren der Mathematischen
  Wissenschaften, vol. 322, Springer-Verlag, Berlin, 1999.

\bibitem{SV2/14}
M.~M. Schein and C.~Voll, \emph{{Normal zeta functions of the Heisenberg groups
  over number rings II -- the non-split case}}, Israel J.\ Math., to
  appear.

\bibitem{Stanley/91}
R.~P. Stanley, \emph{{$f$}-vectors and {$h$}-vectors of simplicial posets}, J.
  Pure Appl. Algebra \textbf{71} (1991), no.~2-3, 319--331.

\bibitem{Stanley/96}
\bysame, \emph{Combinatorics and commutative algebra}, Birkh\"{a}user, 1996,
  second edition.

\bibitem{Stanley/99}
\bysame, \emph{Enumerative combinatorics. {V}ol. 2}, Cambridge Studies in
  Advanced Mathematics, vol.~62, Cambridge University Press, Cambridge, 1999.

\bibitem{Stanley/12}
\bysame, \emph{Enumerative combinatorics. {V}ol. 1}, Cambridge Studies in
  Advanced Mathematics, vol.~49, Cambridge University Press, Cambridge, 2012,
  Second edition.

\bibitem{StasinskiVoll/14}
A.~Stasinski and C.~Voll, \emph{{Representation zeta functions of nilpotent
  groups and generating functions for Weyl groups of type $B$}}, Amer. J. Math.
  \textbf{136} (2014), no.~2, 501--550.

\bibitem{Taylor/01}
G.~Taylor, \emph{Zeta functions of algebras and resolution of singularities},
  PhD~thesis, University of Cambridge, 2001.

\bibitem{Voll/05}
C.~Voll, \emph{{Functional equations for local normal zeta functions of
  nilpotent groups}}, Geom. Func. Anal. (GAFA) \textbf{15} (2005), 274--295,
  with an appendix by A. Beauville.

\bibitem{Voll/10}
\bysame, \emph{{Functional equations for zeta functions of groups and rings}},
  Ann. of Math. (2) \textbf{172} (2010), no.~2, 1181--1218.

\end{thebibliography}
\end{document}